\documentclass[11pt,a4paper,reqno,svgnames]{amsart}

  \usepackage{mathtools}
 
 \usepackage{enumitem}

 \usepackage{amsmath,amsthm,amsfonts,amssymb} 
\usepackage{amsxtra}
 
\usepackage{mathrsfs}
 \usepackage{bm, stmaryrd}
  \usepackage{scalefnt}
\usepackage{latexsym}

\usepackage{xcolor}
\usepackage[plainpages=false, pdfpagelabels,  colorlinks=true, pdfstartview=FitV, linkcolor=DarkBlue, citecolor=DarkRed, urlcolor=blue]{hyperref}

\newcommand{\myref}[2]{\hyperref[#2]{#1~\ref*{#2}}}
\newcommand{\myrefnospace}[2]{\hyperref[#2]{#1\ref*{#2}}}

\usepackage[color]{changebar}
\cbcolor{red}  

\usepackage{comment}   

 \DeclareMathOperator{\A}{A}
 
 \DeclareMathOperator{\Z}{\mathbb{Z}}
 \DeclareMathOperator{\Hom}{Hom}
\DeclareMathOperator{\End}{End}

\DeclareMathOperator{\Res}{Res}
\DeclareMathOperator{\Span}{Span}

\newcommand{\la}{\lambda}
\newcommand{\QQ}{{\mathbb{F}}}
\newcommand{\ZZ}{{\mathbb{Z}}}

\newcommand{\FF}{\mathbb F}

\renewcommand{\ge}{\geqslant}

\renewcommand{\ge}{\geqslant}
\renewcommand{\geq}{\geqslant}
\renewcommand{\le}{\leqslant}
\renewcommand{\leq}{\leqslant}
\renewcommand{\unrhd}{\trianglerighteqslant}

\renewcommand{\succeq}{\succcurlyeq}

 \newcommand{\Std}{{\rm Std}}

\newtheorem{thm}{Theorem}[section]
\newtheorem{cor}[thm]{Corollary}

\newtheorem{lem}[thm]{Lemma}
\newtheorem{prop}[thm]{Proposition}
\newtheorem*{prop*}{Proposition}
\newtheorem*{thm*}{Main Theorem}

\newtheorem*{cor*}{Corollary}
\newtheorem*{conj*}{Conjecture}

\theoremstyle{remark}

\newtheorem{rmk}[thm]{Remark}
\newtheorem{notation}[thm]{Notation}
\newtheorem{rem}[thm]{Remark}

\newtheorem*{Acknowledgements*}{Acknowledgements}

\theoremstyle{definition}
\newtheorem{defn}[thm]{Definition}
\newtheorem{eg}[thm]{Example}

\newtheorem*{ack}{Acknowledgements}

\newcommand{\mft}{\mathsf t}
\newcommand{\mfs}{\mathsf s}
\newcommand{\mfu}{\mathsf u}
\newcommand{\mfv}{\mathsf v}
\newcommand{\stt}{{\mathsf t}}
\newcommand{\sts}{{\mathsf s}}
\newcommand{\stu}{\mathsf u}
\newcommand{\stv}{\mathsf v}

\newcommand{\stw}{\mathsf w}

\newcommand{\Sym}{\mathfrak{S}}
\newcommand{ \Ind}{{\rm Ind}}
\renewcommand{\Res}{{\rm Res}}
\newcommand{\suchthat}{\;\ifnum\currentgrouptype=16 \middle\fi|\;} 
\newcommand\z[2]{z_{#1}^{#2}}   

\renewcommand{\c}{m}
\newcommand{\M}{m}

\def\power #1{^{(#1)}}

\def\leftbrace{\left\{}
\def\rightbrace{\right\}}
\def\ignore#1{\relax}

\def\deltabold{{\bm \delta}}

\def\qbold{{\bm q}}

\newcommand\boldq{\qbold}

\newcommand\cell[3]{\Delta_{#1}^{#2}(#3)}

\title[ Diagram algebras]
{Diagram algebras, \\ dominance triangularity, \\  and skew cell modules}
 \author{Christopher Bowman}
 \author{John Enyang}
 \author{Frederick Goodman}

 
\begin{document}

\maketitle

\begin{abstract}  We present an abstract framework for the axiomatic study of diagram algebras. 
 Algebras that fit this framework possess analogues of both the Murphy  and   seminormal bases of the Hecke algebras of the symmetric groups. 
  We show that  the transition matrix between these bases is dominance unitriangular.
 We  construct analogues  of the skew Specht modules in this setting.  
 This allows us to propose a natural 
tableaux theoretic   framework in which to study the infamous Kronecker problem.   
  \end{abstract}

\section{Introduction}  \label{section introduction}

  The   purpose of this article is to develop an Okounkov--Vershik-style framework in which to study towers of diagram algebras 
  $(A_r)_{r\geq 0}$ 
 over  an integral domain, $R$, and its field of fractions, $\mathbb{F}$.  
 The diagram algebras that fit into  our framework include
 group algebras of the symmetric groups and their Hecke algebras,  the Brauer and BMW algebras, walled Brauer algebras, Jones--Temperley--Lieb algebras, as well as centralizer algebras for the general linear, orthogonal, and symplectic groups acting on tensor spaces.

Following \cite{EG:2012}, 
we observe that  algebras   fitting into our framework possess analogues of both the Murphy  and   seminormal bases of the Hecke algebras of the symmetric groups. 
  We prove that  the transition matrix between these bases is dominance unitriangular.
    In case the algebras have Jucys--Murphy elements, we prove that  these elements act diagonally on the seminormal basis and triangularly with respect to dominance order on the Murphy basis.  
    For the Hecke algebras of symmetric groups, this provides a new and very simple proof of dominance triangularity of the Jucys--Murphy elements (see  \cite[Theorem 4.6]{MR1194316},  \cite[Theorem 3.32 and Proposition 3.35]{MR1711316}).     For other diagram algebras such as the  Brauer algebras, walled Brauer algebras,   BMW algebras or partition algebras, dominance triangularity is a new  result.  
Dominance triangularity is an extremely useful structural  result,   which has already found two distinct applications which we highlight below.  

There is a deeper structure of the Murphy cellular bases of diagram algebras which underlies dominance triangularity.   Elements of the Murphy basis can be written using an ordered product of certain ``branching factors".  There exist both ``down"  and ``up" branching factors and a compatibility relation between them.  Using this compatibility relation, one obtains a certain factorizability property of the Murphy basis elements.    The compatibility and factorizability properties lead to our dominance triangularity results, but also to strong results about restrictions of cell modules, the construction of skew cell modules,  and in \cite{BEG} to a construction of cellular basis for quotients of diagram algebras.   The compatibility and factorizability properties were already observed in ~\cite{EG:2012}, but they are first exploited systematically in this paper and in \cite{BEG}.

The results of this paper  play a crucial role  in \cite{BEG}, where we 
 construct  new integral Murphy-type cellular bases of the   Brauer algebras which  
  decompose into bases for the kernels and images of these algebras acting on tensor space. 
 This construction thus provides simultaneously an integral cellular basis of the centralizer algebra, and a new version of the second fundamental theorem of invariant theory.   
 All of these results are compatible with reduction from $\ZZ$ to an arbitrary field  
(characteristic 2 is excluded in the orthogonal case).

Given two fixed points $\lambda$ and $\nu$ in the $s$th and $r$th levels of the branching graph, we provide an explicit construction of 
an associated {\sf skew cell module} 
$\cell {r-s} {} {\nu \setminus \lambda}$.  
We show that these skew cell modules 
possess  integral  bases  indexed by skew tableaux (paths between the two fixed vertices in the graph) exactly as in the classical case of the symmetric group.   

In the case of the partition algebra, these skew modules provide a new setting in which to study the infamous Kronecker problem.  
In an upcoming paper  \cite{BDE2015}, the first two authors and Maud De Visscher use these skew modules to provide a uniform combinatorial  
interpretation for one of the largest  sets of Kronecker coefficients considered to date  (the Littlewood--Richardson coefficients and the Kronecker coefficients labelled by two 2-line partitions are covered as important examples).

\section{Diagram algebras}  \label{section diagram algebras}

For the remainder of the paper, we shall   let $R$ be an integral domain with field of fractions $\mathbb{F}$. 
In this section, we shall define  diagram algebras and recall the construction of their Murphy bases, following \cite{EG:2012}.  We first recall the definition of a cellular algebra, as in \cite{MR1376244}. 

\subsection{Cellular algebras}
\begin{defn}\label{c-d}
Let $R$ be an integral domain. A {\sf cellular algebra} is a tuple $(A,*,\widehat{A},\unrhd, \Std(\cdot), \mathscr{A})$ where
\begin{enumerate}[label=(\arabic{*}), ref=\arabic{*},leftmargin=0pt,itemindent=1.5em, series= cellular defn]
\item $A$ is a unital $R$--algebra and $*:A\to A$ is an algebra  involution, that is, an $R$--linear anti--automorphism of $A$ such that $(x^*)^* = x$ for $x \in A$;
\item $(\widehat{A},\unrhd)$ is a finite partially ordered set, and  for each $\lambda \in \widehat A$, $\Std(\lambda)$ is a finite indexing set;
\item The set  
$\mathscr{A}=\big\{c_\mathsf{st}^\lambda  \ \big | \  \text{$\lambda\in\widehat{A}$ and $\mathsf{s},\mathsf{t}\in\Std(\lambda)$}\big\}$
is an $R$--basis for $A$.
\end{enumerate}
\noindent
  Let  $A^{\rhd\lambda}$ denote the $R$--module with basis 
  $\big\{
c^\mu_\mathsf{st}\ \mid \mu\rhd\lambda    \text{ and } \mathsf{s},\mathsf{t}\in\Std(\mu) 
\big\}.$
\begin{enumerate}[label=(\arabic{*}), ref=\arabic{*},leftmargin=0pt,itemindent=1.5em, resume= cellular defn]

\item The following  two conditions hold for the basis $\mathscr A$.

\begin{enumerate}[label=(\alph{*}), ref=\alph{*},itemindent=1.5em]
\item\label{c-d-1} Given $\lambda\in\widehat{A}$, $\mathsf{t}\in\Std(\lambda)$, and $a\in A$, there exist 
coefficients $r(  a; \mathsf{t}, \mathsf{v}) \in R$, for $\mathsf{v}\in\Std(\lambda)$, such that, for all $\mathsf{s}\in\Std(\lambda)$, 
\begin{align}\label{r-act}
c_\mathsf{st}^\lambda a\equiv 
\sum_{\mathsf{v}\in\Std(\lambda)}
r( a; \mathsf{t}, \mathsf{v}) c_{\mathsf{sv}}^\lambda \mod{A^{\rhd\lambda}},
\end{align}

\item\label{c-d-2} If $\lambda\in\widehat{A}$ and $\mathsf{s},\mathsf{t}\in\Std(\lambda)$, then $(c_\mathsf{st}^\lambda)^*  \equiv (c_{\stt\sts}^\lambda)  \mod A^{\rhd \lambda}$.
\end{enumerate}
\end{enumerate}
The tuple $(A,*,\widehat{A},\unrhd, \Std(\cdot), \mathscr{A})$ is a  {\sf cell datum} for $A$. 
   The basis $\mathscr{A}$ is called a {\sf  cellular basis} of $A$.  
\end{defn}
 
 If $A$ is a cellular algebra over $R$, and $R \to S$ is a homomorphism of integral domains, then  the specialization $A^S = A\otimes_R S$ is a cellular algebra over $S$, with
 cellular basis 
 $$\mathscr A^S = \{c^\lambda_{\sts\stt}	\otimes 1_S	\mid  \lambda\in\widehat{A},  \text{ and }\mathsf{s},\mathsf{t}\in\Std(\lambda) \}.$$
 In particular, $A^\FF$ is a cellular algebra.     Since the map $a \mapsto a \otimes 1_\FF$ is injective, we regard $A$ as contained in $A^\FF$ and  we identify $a \in A$ with $a \otimes 1_\FF \in A^\FF$.  
 
An order ideal $\Gamma \subset \widehat A$ is a subset with the property that if $\la \in \Gamma$ and 
$\mu \unrhd \la$, then $\mu \in \Gamma$.  It follows from the axioms of a cellular algebra that for any order ideal $\Gamma$ in $\widehat A$, 
$$
A^\Gamma = \Span \big\{c^\la_{\mfs \mft}  \ \big \vert \   \la \in \Gamma, \mfs, \mft \in  \Std(\lambda) \big\}
$$
is an involution--invariant two sided ideal of $A$.  In particular $A^{\rhd \la}$ and
$$
A^{\unrhd \la} = \Span\ \big\{
c^\mu_{\sf st}\ \big \vert \  \text{$\mu\in\hat{A}$, $\Sym,{\sf t}\in \Std(\mu)$ and $\mu\unrhd\lambda$}
\big\}
$$
are involution--invariant two sided ideals.

\begin{defn} \label{definition: cell module}
Let $A$ be a cellular algebra over $R$ and $\lambda\in\hat{A}$. The  {\sf cell module} $\cell  {} {}  \lambda$  is the right $A$--module defined as follows.  As an $R$--module, $\cell  {} {} \lambda$ is free with basis indexed by $\Std(\lambda)$,  say $\{c^\la_\mft  \mid \mft \in \Std(\la) \}$.  
The right $A$--action is given by 
$$
c_\mft^\la a =  \sum_{{{\sf v}\in\hat{A}^\lambda}} r( a; \mathsf{t}, \mathsf{v})  c^\la_\mfv,
$$
where the coefficients $r( a; \mathsf{t}, \mathsf{v})$  are those of Equation~\eqref{r-act}.
\end{defn}

Thus, for any $\mfs \in \widehat A^\la$,  
$$
\Span\{ c^\la_{\mfs \mft} + A^{\rhd \la} \mid  \mft \in \Std(\lambda)\} \subseteq A^{\unrhd \la}/A^{\rhd \la}
$$ 
is a model for the cell module $\cell {}{}\la$. 
When we need to emphasize the algebra or the ground ring, we may write $\cell A {} \lambda$ or 
 $\cell {} R \lambda$.   Note that $\cell {} \FF \lambda = \cell {} {} \lambda \otimes_R \FF$ is the cell module for $A^\FF$ corresponding to $\lambda$.  
 
If $A$ is an $R$--algebra with involution $*$,  then $*$ induces functors $M \to M^*$   interchanging left and right $A$--modules, and taking $A$--$A$ bimodules to $A$--$A$ bimodules.   We identify $M^{**}$ with 
$M$ via $x^{**} \mapsto x$  and for modules ${}_A M$ and $N_A$ we have 
$
(M \otimes_R N)^* \cong  N^* \otimes_R M^*,  
$
as $A$--$A$ bimodules, with the isomorphism determined by $(m \otimes n)^* \mapsto n^* \otimes m^*$. 
For a right $A$--module $M_A$, using both of these isomorphisms, we identify 
$(M^* \otimes M)^*$ with $M^{*} \otimes M^{**}  = M^* \otimes M$, via 
$(x^* \otimes y)^*  \mapsto y^* \otimes x$.
Now we apply these observations with $A$ a cellular algebra and $\cell {} {} \la$ a cell module.  The assignment $$\alpha_\la : c^\la_{\mfs \mft}  + A^{\rhd \la}  \mapsto  (c^\la_\mfs)^* \otimes (c^\la_\mft)$$ determines an
$A$--$A$ bimodule isomorphism from $A^{\unrhd\la}/\A^{\rhd\la}$ to $(\cell {} {} \la)^* \otimes_R \cell {}  {} \la$. Moreover,
we have   $*\circ \alpha_\la =  \alpha_\la \circ *$, which reflects the cellular algebra axiom
$(c^\la_{\mfs \mft})^* \equiv c^\la_{\mft \mfs} \mod  A^{\rhd \la}$.  

\def\inv{^{-1}}

A certain bilinear form on the cell modules plays an essential role in the theory of cellular algebras.    Let $A$ be a cellular algebra over $R$ and let $\lambda \in \widehat A$.   The cell module $\cell {}{}\la$ can be regarded as an $A/A^{\rhd \la}$ module.
 For
$x, y, z \in \cell {}{}\lambda$, it follows from the the definition of  the cell module and the map $\alpha_\lambda$ that $x \alpha_\la\inv(y^* \otimes z) \in R z$.  Define $\langle x, y \rangle$ by
\begin{equation} \label{defn of bilinear form}
x \alpha_\la\inv(y^* \otimes z)  =  \langle x, y \rangle z.
\end{equation}
Then $ \langle x, y \rangle$  is $R$-linear in each variable and we have $\langle x a, y\rangle =
\langle x,  y a^* \rangle$  for $x, y \in \cell {}{}\lambda$ and $a \in A$.    Note that
$$
c^\lambda_{\mfs \mft}  c^\lambda_{\mfu \mfv} = \langle c^\la_\mft, c^\la_\mfu \rangle c^\la_{\mfs \mfv},
$$
which is the customary definition of the bilinear form.

\begin{defn}[\mbox{\cite{MR3065998}}]
A cellular algebra, $A$,  is said to be {\sf cyclic cellular} if every cell module  is cyclic as an $A$-module.   \end{defn}

If $A$ is cyclic cellular, $\la \in \widehat A$, and $\delta(\la)$ is a generator of the cell module $\cell {}{}\la$,
let $\c_\la$ be a lift in $A^{\unrhd \la}$ of $\alpha_\la\inv(\delta(\la)^* \otimes \delta(\la))$.  

\begin{lem} \label{properties of c lambda}
The element
$\c_\la$ has the following properties:
\begin{enumerate}[label=(\arabic{*}), font=\normalfont, align=left, leftmargin=*]
\item \label{cyclic gen a} $\c_\la \equiv \c_\la^* \mod{A^{\rhd \la}}$.
\item  \label{cyclic gen b} $A^{\unrhd \la} =  A \c_\la A +  A^{\rhd \la}$.
\item   \label{cyclic gen c}  $(\c_\la A + A^{\rhd \la})/A^{\rhd \la} \cong \cell {} {} \la$, as right $A$--modules. 
\end{enumerate}
\end{lem}

\begin{proof}  Lemma 2.5 in ~\cite{MR3065998}.
\end{proof}

\noindent In examples of interest to us, we can always choose $\c_\la$ to satisfy $\c_\la^* = \c_\la$ (and moreover, $\c_\la$ is given explicitly).

\ignore{
\begin{defn}
A cellular algebra $A$ over the integral domain $R$ is said to be {{\sf generically semisimple}} if 
$A^\FF=A\otimes_R \FF$ is semisimple.   
\end{defn}
}

\subsection{Sequences of diagram algebras}  \label{subsection diagram algebras}
 Here and in the remainder of the paper, we will consider    an   increasing sequence   $(A_r)_{r \ge 0}$
of  cellular  algebras over an integral domain $R$ with field of fractions $\FF$.   
We assume that all the inclusions are unital and that the involutions are consistent; that is the involution on $A_{r+1}$, restricted to $A_r$, agrees with the involution on $A_r$.      We will establish a list of assumptions \eqref{diagram 1}--\eqref{diagram compatibility}.  For convenience, we call an increasing sequence of  cellular algebras satisfying these assumptions a {\sf sequence of diagram algebras}. 

Let $(\widehat A_r, \unrhd)$ denote the partially ordered set in the cell datum for $A_r$.   For $\lambda \in \widehat A_r$,  let $\Delta_r(\lambda)$  denote the corresponding  cell module.    If $S$ is an integral domain with a unital homomorphism $R \to S$,  write $A_r^S = A_r \otimes_R S$  and $\Delta_r^S(\lambda)$  for $\Delta_r(\lambda) \otimes_R S$.    In particular, write 
 $A_r^\FF = A_r \otimes_R \FF$  and $\Delta_r^\FF(\lambda)$  for $\Delta_r(\lambda) \otimes_R \FF$.

 \begin{defn} \label{definition: cell-filtration}  Let $A$ be a cellular  algebra over $R$.  
 If $M$ is a right $A$--module, a  {\sf  cell-filtration}  of  $M$  is a filtration by right $A$--modules
\begin{align*}
\{ 0\}= M_0  \subseteq   M_1  \subseteq \cdots  \subseteq    M_s=M,
\end{align*}
such that $  M_{i} / M_{i-1} \cong  \cell {} {} {\lambda\power i}$ for some $\lambda^{(i)} \in \widehat{A}$.   We  say that the filtration is {\sf  order preserving} if  $\lambda^{(i)}\rhd \lambda^{(i+1)}$ in
  $\widehat{A}$  for all $i\geq 1$.  
 \end{defn}

 \begin{defn}[\cite{MR2794027,MR2774622}]  Let  $(A_r)_{r \ge 0}$ be an increasing sequence of cellular algebras over an integral domain $R$.
  \begin{enumerate}[label=(\arabic{*}), font=\normalfont, align=left, leftmargin=*] 
  \item  The tower $(A_r)_{r \ge 0}$ is {\sf  restriction--coherent} if 
  for each $r \ge 0$ and each   $\mu \in \widehat{A}_{r+1}$,    the restricted module $\Res_{A_r}^{A_{r+1}}(\Delta_{r+1}(\mu))$  has an order preserving cell-filtration.  
  \item   A tower  $(A_r)_{r \ge 0}$ is {\sf  induction--coherent} if 
for each $r\ge 0$ and each  $\lambda \in \widehat{A}_r$,  the induced module $\Ind_{A_r}^{A_{r+1}}(\Delta^R_r(\la))$   has an order preserving  cell-filtration. 
\item The tower  $(A_r)_{r \ge 0}$ is {\sf  coherent} if it is both restriction-- and induction--coherent.
\end{enumerate}
 \end{defn}
 
 \begin{rem}  We have changed the terminology from ~\cite{EG:2012,MR2794027,MR2774622}, as the weaker notion of coherence, in which the order preserving requirement is omitted, plays no role here. 
 \end{rem}

\noindent We now list the first of our assumptions for a sequence of diagram algebras:
 \begin{enumerate}[label=(D\arabic{*}), ref=D\arabic{*},  series = DiagramAlgebras]
\item  \label{diagram 1}  $A_0 = R$.
 \item   \label{diagram 2}  \label{diagram cyclic cellular}  The algebras $A_r$ are cyclic cellular  for all $r \geq 0$.  
 \end{enumerate}
 
 For all $k$ and for all $\lambda \in \widehat A_k$, fix once and for all a bimodule isomorphism
$\alpha_\lambda: A_k^{\unrhd \la}/A_k^{\rhd \la} \to (\cell k {} \lambda)^* \otimes_R \cell k {} \lambda$, 
a generator $\delta_k(\la)$ of the cyclic $A_k$--module $\cell k {} \lambda$,  
and an element $\M_\lambda \in A_k^{\unrhd \la}$  satisfying  $\alpha_\la(\M_\la + A_k^{\rhd \la}) = (\delta_k(\la))^* \otimes \delta_k(\la)$.    Recall the properties of $\M_\la$ from
\myref{Lemma}{properties of c lambda}.  We require the following mild assumption on the elements $m_\lambda$.

 \begin{enumerate}[label=(D\arabic{*}), ref=D\arabic{*},  resume = DiagramAlgebras]
 \item  \label{diagram symmetry}  \label{diagram 5}
 $\M_\la = \M_\la^*$.
 \end{enumerate}
  Our list of assumptions continues as follows:
   \begin{enumerate}[label=(D\arabic{*}), ref=D\arabic{*},  resume = DiagramAlgebras]
 \item  \label{diagram 3}  \label{diagram semisimplicity}
 $A_r^\FF$ is split semisimple for all $r\geq 0$.  
\item  \label{diagram 4} \label{diagram restriction coherent}
 The sequence of algebras  $(A_r)_{r \ge 0}$  is  restriction--coherent.  
\end{enumerate}

As discussed in ~\cite[Section 3]{EG:2012},  under the assumptions \eqref{diagram 1}--\eqref{diagram restriction coherent}
above, there exists a well defined  
multiplicity--free branching diagram $\widehat A $ associated with the sequence $(A_r)_{r \ge 0}$.  The branching diagram is an infinite, graded, directed graph with vertices $\widehat A_r$ at level $k$ and edges determined as follows. 
For $\lambda \in \widehat A_{r-1}$ and $\mu \in \widehat A_r$, there is an edge $\lambda \to \mu$ in $\widehat A$ if and only if $\Delta_{r-1}(\lambda)$  appears as a subquotient of an order preserving cell filtration of  $\Res^{A_{r}}_{A_{r-1}} (\Delta_r(\mu))$.   Note that $\widehat A_0$ is a singleton;  we denote   its unique element by $\varnothing$.   
 We can choose $\cell 0 {} \varnothing = R$,  $\delta_0(\varnothing) = 1$, and $\M_\varnothing = 1$.

  \begin{defn}
Given  $\nu \in \widehat{A}_{r}$, we  
define a  {\sf  standard   tableau} of shape $ \nu$  to be a 
directed path  $\stt$  on the branching diagram $\widehat A$ from $\varnothing \in \widehat A_0$ to $\nu$, 
 $$
\stt = (\varnothing= \stt(0) \to \stt(1) \to    \stt(2)\to   \dots \to  \stt(r-1)\to \stt(r) = \nu). 
 $$ 
We   let $\Std_r(\nu ) $     denote the set of all such paths and let 
$\Std_{r}= \cup_{ \nu \in \widehat{A}_r}\Std_{r}(\nu )$.    \end{defn}

It is shown in in ~\cite[Section 3]{EG:2012} that 
there exist certain ``branching factors" 
$d_{\lambda \to \mu} \in A_r$ associated to each edge $\lambda \to \mu$ in $\widehat A$, related to the cell filtration of $\Res^{A_{r}}_{A_{r-1}} (\Delta_r(\mu))$.  Given a  path
$\mft \in \Std_{r}(\nu )$, 
  $$
\varnothing = \stt(0) \to \stt(1) \to    \stt(2)\to   \dots \to  \stt(r-1)\to \stt(r) = \nu, 
 $$
 define
$
d_\mft =  d_{\stt({r-1}) \to \stt( r)}  d_{ \stt( {r-2}) \to \stt( {r-1})} \cdots  d_{\stt(0) \to \stt( 1)}.
$

We say two cellular bases of an algebra $A$ with involution are {\sf equivalent} if they determine the same
two sided ideals $A^{\unrhd \lambda}$  and isomorphic cell modules.

\begin{thm}[\cite{EG:2012},  Section 3] \label{theorem abstract Murphy basis}  
Let $(A_r)_{r \ge 0}$ be a sequence of algebras satisfying assumptions \eqref{diagram 1}--\eqref{diagram restriction coherent}. 
\begin{enumerate}[label=(\arabic{*}), font=\normalfont, align=left, leftmargin=*]
\item Let $\lambda \in \widehat A_r$.  The set  $\{ \c_\lambda d_\mft  +  A_r^{\rhd \lambda} \suchthat    \mft \in \Std_r(\lambda)  \}$   is a basis of the cell module $\Delta_r(\lambda)$. 
\item The set  $\{ d_\mfs^* \c_\lambda d_\mft \suchthat  \lambda \in \widehat A_r,  \mfs, \mft \in \Std_r(\lambda) \} $ is  a cellular basis of $A_r$, equivalent to the original cellular basis. 
\item  For a fixed $\lambda \in \widehat A_r$, we let $\mu(1) \rhd \mu(2) \rhd \cdots  \rhd \mu(s)$ be a listing of the $\mu \in \widehat A_{r-1}$ such that $\mu \to \lambda$.    Let
$$
M_j = \Span_R \leftbrace  \c_\lambda d_\mft +  A_r^{\rhd \lambda} \suchthat  \mft \in \Std_r(\lambda), \mft(k-1) \unrhd \mu(j)                         \rightbrace.
$$
Then
$$
(0) \subset  M_1 \subset  \cdots \subset  M_s = \Delta_r(\lambda)
$$
is a filtration of $\Delta_r(\lambda)$ by $A_{r-1}$ submodules, and $M_j/M_{j-1} \cong \Delta_{r-1}(\mu_j)$.  
\end{enumerate}
\end{thm}

\begin{notation}
We write $m_{\mfs  \mft}^\lambda =   d_\mfs^*  \c_\lambda d_\mft$.    Also write
$m_\mft =  \c_\lambda d_\mft + A_r^{\rhd \lambda} \in \Delta_r(\lambda)$.   We refer to the cellular basis
$\{m_{\mfs  \mft}^\lambda  \suchthat  \la \in \widehat A_r \text{ and }  \mfs, \mft \in \Std_r(\la)\}$ as the {\sf Murphy  cellular basis} of $A_r$  and $\{m^\la_\mft  \suchthat \mft \in \Std_r(\la) \}$  as the 
{\sf Murphy  basis} of the cell module $\cell k {} \la$. 
\end{notation}

We will now continue with our list of assumed properties of the sequence of algebra 
$(A_r)_{r \ge 0}$ with one final axiom.  
 
 \begin{enumerate}[label=(D\arabic{*}), ref=D\arabic{*},  resume=DiagramAlgebras]
 \item  \label{diagram 7} \label{diagram compatibility}
 There exist   `$u$-branching factors' $u_{\mu \to \lambda}\in A_{r+1}^R$ such that 
 \begin{equation} \label{abstract branching compatibility}
 \c_\lambda d_{\mu \to \lambda} =   (u_{\mu \to \lambda})^*  \c_\mu. 
\end{equation}
\end{enumerate}

 \begin{eg}\label{theoneexample}  It is shown in ~\cite{EG:2012} that the Hecke algebras of type $A$, the symmetric group algebras,  the   Brauer algebras, the Birman--Wenzl--Murakami algebras,  the partition algebras, and the Jones--Temperley--Lieb algebras  all are examples of sequences of algebras satisfying properties \eqref{diagram 1}--\eqref{diagram compatibility}.  For the Hecke algebras, the cellular basis of 
 \myref{Theorem}{theorem abstract Murphy basis} agrees with the Murphy basis from \cite{MR1327362}, up to a normalization.
In each case the ground ring $R$ can be taken to be the generic ground ring for the class of algebras.  For example, for the Hecke algebras, this is $\Z[\boldq, \boldq\inv]$, and for the Brauer algebras it is $\Z[\deltabold]$, where  $\boldq$ and   $\deltabold$ are indeterminants.

Let  $G_n$ be either the general linear group ${ \sf GL}_n$, the orthogonal group $ {\sf O}_n$, or the symplectic group ${ \sf Sp}_{2n}$   and let $V$ denote its natural module.  
 The centralizer algebra $\End_{G_n}(V^{\otimes r})$ is a quotient of the symmetric group $\mathfrak{S}_r$, the Brauer algebra $B_r(n)$, or 
the Brauer algebra $B_r(-2n)$ respectively.   
In each case the ground ring $R$ can be taken to be $\ZZ$.  
It is shown in \cite{BEG} that the algebras $\End_G(V^{\otimes r})$ all satisfy axioms  \eqref{diagram 1}--\eqref{diagram compatibility}.  
\end{eg}

  \begin{defn}
Given $0\leq s \leq r$ and    $\lambda  \in   \widehat{A}_{s}$, $\nu \in \widehat{A}_{r}$, we  
define a  {\sf  skew  standard   tableau} of shape $ \nu \setminus  \lambda  $ and degree $r-s$  to be a directed path $\stt$  on the branching diagram $\widehat A$ from $\lambda$ to $\nu$, 
 $$
\stt = (\lambda = \stt(s) \to \stt(s+1) \to    \stt(s+2)\to   \dots \to  \stt(r-1)\to \stt(r) = \nu). 
 $$
We let $\Std_{s,r}(\nu \setminus  \lambda)$ denote the set of all such paths with given $\lambda$ and $\nu$. 
 Given $0\leq s \leq r$, we set $\Std_{s,r}= \cup_{\la \in \widehat{A}_s,\nu \in \widehat{A}_r}\Std_{s,r}(\nu\setminus\lambda)$.  
 \end{defn}

 Given two paths $\sts \in \Std_{q, s}(\mu \setminus \lambda)$ and $\stt \in \Std_{s, r}(\nu \setminus \mu)$ such that the final point of $\sts$ is the initial point of $\stt$,  define $\sts \circ \stt$ to be the obvious path obtained by concatenation.
 
 \begin{rmk} \label{remark factorization}
 Given a  path
$\mft \in \Std_{s,r}(\nu \setminus  \lambda)$  of the form  
 $$
\lambda = \stt(s) \to \stt(s+1) \to    \stt(s+2)\to   \dots \to  \stt(r-1)\to \stt(r) = \nu, 
 $$
define
\begin{align*}
d_\mft &=  d_{\stt({r-1}) \to \stt( r)}  d_{ \stt( {r-2}) \to \stt( {r-1})} \cdots  d_{\stt(s) \to \stt(s+ 1)},  \\
\intertext{and}
u_\mft  &=     u_{\stt(s) \to \stt( s+1)}  \cdots
u_{ \stt( {r-2}) \to \stt( {r-1})} u_{\stt({r-1}) \to \stt( r)}.
\end{align*}
  Then it follows from the compatibility relation \eqref{abstract branching compatibility} and induction on $r -s$ that
\begin{equation} \label{abstract branching compatibility for path}
u_\mft^*  \c_\la =  \c_\nu  d_\mft.
\end{equation}
Because $\c_\varnothing$ can be chosen to be $1$,  this gives in particular for $\mft \in \Std_r(\nu)$, 
\begin{equation} \label{abstract branching compatibility for path 2}
u_\mft^*  =  \c_\nu  d_\mft.
\end{equation}
Therefore the cellular basis $\{ m_{\mfs \mft}^\nu\}$  can also be written in the apparently asymmetric form
$$
 m_{\mfs  \mft}^\nu = d_\mfs^* \c_\nu d_\mft =  d_\mfs^*  u_\mft^*.  
$$
Using the symmetry of the cellular basis  $(m_{\mfs  \mft}^\nu)^* = m_{\mft \mfs}^\nu$ (which follows from the assumption \eqref{diagram symmetry}), we also get
$$
 m_{\mfs  \mft}^\nu = u_\mfs  d_\mft.
$$
Using \eqref{abstract branching compatibility for path 2}, we have the following form for the basis $\{m^\nu_\mft \suchthat  \mft \in \Std_r(\nu) \}$ of the cell module $\cell r {} \nu$:
\begin{equation}\label{abstract branching compatibility for path 3}
m^\nu_\mft =  u^*_\mft + A_r^{\rhd \nu}.
\end{equation}
Now, for any $0\leq q \le s \le r$,   let $\mft_{[q, s]}$ denote the truncated path,
$$
 \stt(q) \to \stt(q+1) \to    \stt(q+2)\to   \dots \to  \stt(s-1)\to \stt(s). 
 $$
The representative $u^*_\mft$ of $m_\mft$  has the remarkable property that  for any $0\leq s \le r$,
\begin{equation} \label{eqn factorization of u t}
u^*_\mft =  u^*_{\mft_{[s, r]}}   u^*_{\mft_{[0, s]}},
\end{equation}
and 
\begin{equation}  \label{eqn property of u t trunacation}
 u^*_{\mft_{[0, s]}} = \c_{\mft(s)} d_{\mft_{[0, s]}} \in  \c_{\mft(s)} A_s \subseteq A_s^{\unrhd \mft(s)}.
 \end{equation}
 
  The compatibility relations \eqref{abstract branching compatibility for path} together with the factorizability 
\eqref{eqn factorization of u t}  
 of representatives $u^*_\stt$  of the Murphy basis play a crucial role in this paper.  They lead directly to our dominance triangularity results, which in turn lead to strong results about restriction of cell modules and allow the construction of skew cell modules.  In our view, these are the distinguishing properties of the Murphy bases of diagram algebras, and even in the original context of the Hecke algebras ~\cite{MR1327362} these properties provide  new insight. 
 \end{rmk}

\section{Seminormal basis and dominance triangularity}  \label{section seminormal}
 \subsection{Gelfand--Zeitlin idempotents} \label{GZ idempotents}
 Consider an increasing sequence of  algebras    $(A_r)_{r \ge 0}$ satisfying assumptions
 \eqref{diagram 1}--\eqref{diagram compatibility} of 
 \myref{Section}{subsection diagram algebras}. 
  Let us recall the following notion pertaining to the tower $(A_r^\FF)_{r \ge 0}$. The terminology is from   Okounkov and Vershik ~\cite{MR1443185,MR2050688}. 
 
 \begin{defn}  The {\em Gelfand--Zeitlin subalgebra} $G_r$  of $A_r^\FF$ is the subalgebra generated by the centers of $A_0^\FF, A_1^\FF, \dots, A_r^\FF$.  
 \end{defn}
 
The  Gelfand--Zeitlin subalgebra is a maximal abelian subalgebra of $A_r^\FF$ and contains a canonical family of idempotents indexed by paths on the branching diagram $\widehat A$.  
For each $s$ let $\{\z s \la \suchthat \la \in \widehat A_s\}$ denote the set of minimal central idempotents in $A_s^\FF$.
For $r \ge 1$ and $\mft$ a path on $\widehat  A$ of length $r$, let $F_\mft = \prod_{s = 1}^r  \z s {\mft(s)}$.   The elements $F_\mft$ for $\mft\in\Std_r$ are mutually orthogonal minimal idempotents whose sum is the identity;  moreover  $\sum_{\stt \in \Std_r(\la)}F_\mft = \z r \la$.  
 If $\mfs\in \Std_s$  and $\mft\in \Std_r$   with $s \le r$, then $F_\mfs F_\mft =  \delta_{\mfs, \mft_{[0, s]}} F_\mft$.   
 Evidently, the set  $\{F_\mft \mid \mft\in\Std_s, 0\leq s \leq r\}$  generates $G_r$.
    Let us call this set of idempotents   the family of {\sf Gelfand--Zeitlin idempotents} for $(A_r^\FF)_{r \ge 0}$.

\subsection{Seminormal bases}  \label{section SN}
Let $\{m_{\mfs \mft}^\la \suchthat  \la \in \widehat A_r \text{ and }  \mfs, \mft \in \Std_r(\la)\}$ be the Murphy cellular basis of $A_r$ constructed in 
\myref{Section}{subsection diagram algebras},
and for $\la \in \widehat A_r$,  let $\{m^\la_\mft \suchthat  \mft \in \Std_r(\la)\}$ be the Murphy basis of the cell module $\cell r {} \la$.    For  $\mft \in \Std_r$,  let $F_\mft$ be the corresponding Gelfand--Zeitlin idempotent in $A_r^\FF$. 

\begin{defn}[Seminormal bases] \label{defn seminormal}
 Let  $r \ge 0$  and $\la \in \widehat A_r$.    For $\mfs, \mft \in \Std_r(\la)$, define
$f^\la_\mft = m_\mft F_\mft \in \cell r \FF \la$   and $F_{\mfs \mft}^\la =  F_\mfs m_{\mfs \mft}^\la  F_\mft \in A_r^\FF$.  
\end{defn}

We now define two partial orders the sets of paths $\Std_{s, r}$ in the branching graph. 

\begin{defn}[Dominance order]  For $\mfs, \mft \in \Std_{s, r}$,   define $\mfs \unrhd \mft$ if $\mfs(j) \unrhd \mft(j)$ for all $0\leq j\leq r$.  We write $\mfs \rhd \mft$ if $\mfs \ne \mft$ and $\mfs \unrhd \mft$.
\end{defn}

This is evidently a partial order, which we call the dominance order on paths.
  In particular, the dominance order is defined on $\Std_r$ and on $\Std_r(\la)$ for $\la \in \widehat A_r$.

\begin{defn}[Reverse lexicographic order]   For $\mfs, \mft \in \Std_{s, r}$,   define $\mfs \succeq \mft$ if  $\mfs = \mft$ or if for the last index $j$ such that $\mfs(j) \ne \mft(j)$, we have
$\mfs(j) \rhd \mft(j)$.  We write  $\mfs \succ \mft$ if $\mfs \ne \mft$ and $\mfs \succeq \mft$.
\end{defn}

This is also a partial order on paths (and is defined on $\Std_r$ and on $\Std_r(\la)$ for $\la \in \widehat A_r$).    
 Evidently $\mfs \rhd \mft$ implies $\mfs \succ \mft$.

\begin{thm}[Dominance triangularity]  \label{dominance triangularity}
 Fix $\la \in \widehat A_r$.     For all $\mft \in \Std_r(\la)$,  there exist coefficients 
 $r_\mfs, r'_\mfs  \in \FF$ such that
$$
m^\la_\mft =  f^\la_\mft  +  \sum_{\begin{subarray}c \mfs \in \Std_r(\la) \\  \mfs \rhd \mft \end{subarray}}  r_\mfs f^\la_\mfs 
\qquad \quad 
f^\la_\mft =  m^\la_\mft  +  \sum_{\begin{subarray}c \mfs \in \Std_r(\la) \\  \mfs \rhd \mft \end{subarray}} r'_\mfs m^\la_\mfs. 
$$
In particular, $\{f^\la_\mft \suchthat \mft \in \Std_r(\la) \}$ is a basis of $\cell r \FF \la$.  
\end{thm}

\begin{proof}
The element $\sum_{\mu \unrhd \lambda}   {\z s \mu}$ acts as the identity on the ideal $  A_s^{\unrhd \lambda}$.  
For $\la \in \widehat A_r$  and $\mft \in \Std_r(\la)$,   $u^*_\mft$ has the property that $u^*_{\mft_{[0, s]}} \in A_s^{\unrhd \mft(s)}$ for all $0\leq s \le r$, and   $u^*_\mft = u^*_{\mft_{[s, r]}}   u^*_{\mft_{[0, s]}} $.  We therefore have that  
$$
u^*_\mft = u^*_\mft   \sum_{\mu(s) \unrhd \mft(s)} {\z s {\mu(s)}} .
$$
Applying this at each $0\leq s \le r$ gives
$$
u^*_\mft = u^*_\mft    \prod_{1 \le s \le r}\sum_{\mu(s) \unrhd \mft(s)} {\z s {\mu(s)}} 
=  u^*_\mft   \sum_{(\mu(1), \mu(2), \dots, \mu(r))}   \prod_{1 \le s \le r}  {\z s {\mu(s)}}
$$
where  the sum is over all sequences $(\mu(1), \mu(2), \dots, \mu(r))$  such that 
$\mu(s) \in \widehat A_s$ and $\mu(s) \unrhd \mft(s)$ for all $0\leq s \leq r$.    If such a sequence is not an element of $\Std_r$, i.e. is not
a path on $\widehat A$, then the product $\prod_{1 \le s \le r}  {\z s {\mu(s)}}$ is zero.  On the other hand, if 
$$
\mfs = \left(\mu(1) \to \mu(2) \to \cdots \to \mu(r)\right)
$$
is an element of $\Std_r$,  then $\prod_{1 \le s \le r}  {\z s {\mu(s)}} = F_\mfs$.    Thus
$$
\begin{aligned}
u^*_\mft &=   u^*_\mft   \sum _{\begin{subarray}c  \mfs \in \Std_r\\  \mfs \unrhd \mft \end{subarray}}
  F_\mfs    = u^*_\mft  F_\mft +   \sum_{\begin{subarray}c  \mfs \in \Std_r(\lambda)\\  \mfs \rhd \mft \end{subarray}}  u^*_\mft  F_\mfs  + y,
\end{aligned}
$$
where $y \in  (A_r^\FF)^{\rhd \la}$.    Passing to the cell module $\Delta^\FF_r(\la)$ we have
$$
m^\la_\mft =  f^\la_\mft  +  \sum_{\begin{subarray}c  \mfs \in \Std_r(\lambda)\\  \mfs \rhd \mft \end{subarray}}  
  m^\la_\mft  F_\mfs.
$$
But the range of $F_s$ acting on the simple module  $\Delta^\FF_r(\la)$ is of dimension 1, spanned by 
$f^\la_\mfs$, so this gives
$$
m^\la_\mft =  f^\la_\mft  +  \sum _
{\begin{subarray}c  \mfs \in \Std_r(\la)\\  \mfs \rhd \mft \end{subarray}}
r_\mfs f^\la_\mfs ,$$
for appropriate $r_\mfs \in \FF$.     This shows that the tuple $[m^\la_\mft]_{\mft \in \Std_r(\la)}$ is related to the tuple $[f^\la_\mft]_{\mft \in \Std_r(\la)}$ by an (invertible)  matrix which is unitriangular with respect to dominance order of paths.    Hence $\{f^\la_\mft \suchthat \mft \in \Std_r(\la)\}$ is a basis of $\cell r \FF \la$ and the inverse change of basis matrix is also unitriangular.  
\end{proof}

\begin{cor}  \label{corollary seminormal basis}  For $r \ge 0$, we have that 
\begin{enumerate}[label=(\arabic{*}), font=\normalfont, align=left, leftmargin=*]
\item
   $\{f^\la_\mft \suchthat \mft \in \Std_r(\la)\}$ is a basis of $\cell r \FF \la$ for all  $\la \in \widehat A_r$.
\item    $\{ F_{\mfs \mft}^\la \suchthat \la \in \widehat A_r \text{ and }  \mfs, \mft \in \widehat A_r\}$  is a cellular basis of $A_r^\FF$.  
\end{enumerate}
\end{cor}

\begin{proof}  The first statement was verified in the proof of 
\myref{Theorem}{dominance triangularity}.
 The second statement follows from ~\cite[Lemma 2.3]{MR2774622}, since $F_{\mfs \mft}^\la$ is a lift of $\alpha_\la\inv(f^\la_\mfs \otimes_\FF f^\la_\mft)$.  
\end{proof}

We note that the cell module $\cell r \FF \la$ imbeds in the algebra $A_r^\FF$ as a right ideal:

\begin{lem} \label{lemma embedding of cell module in algebra}
 Let $r \ge 0$ and let $\la  \in \widehat A_r$.  \mbox{}
\begin{enumerate}[label=(\arabic{*}), font=\normalfont, align=left, leftmargin=*]
\item   $\Span_\FF\{ u^*_\mft F_\mft \suchthat \mft \in \Std_r(\la)\}$ 
 is a right ideal of $A_r^\FF$  and 
$f_\mft^\la \mapsto u^*_\mft F_\mft$ determines an isomorphism of $\cell r \FF \la$ onto this right ideal.
\item Likewise for any fixed  $\mfs \in \Std_r(\la)$,
$$
\cell r \FF \la \cong  \Span_\FF\{m^\la_{\mfs \mft} F_\mft \suchthat \mft \in \Std_r(\la)\}  
\cong   \Span_\FF\{F^\la_{\mfs \mft}  \suchthat \mft \in \Std_r(\la)\},
$$
with isomorphisms determined by $f_\mft^\la \mapsto m^\la_{\mfs \mft} F_\mft$,  respectively
$f_\mft^\la \mapsto F^\la_{\mfs \mft}$.
\end{enumerate}
\end{lem}

\begin{proof}   Recall that $\delta_\la$ is a generator of the cell module $\cell r {\FF} \la$  and
$\c_\la$ is a lift  in $(A_r^\FF)^{\unrhd \la}$ of  $\alpha_\la\inv((\delta_\la)^* \otimes_\FF \delta_\la)$.    For any fixed $x \in A_r^\FF$ such that $\delta_\la x \ne 0$,  $$f^\la_\mft \mapsto   \alpha_\la\inv(x^* (\delta_\la)^* \otimes_\FF  f^\la_\mft ) =  x^* \c_\la d_\mft F_\mft +  (A_r^\FF)^{\rhd \la}$$ determines an isomorphism of 
$\cell k \FF \la$ onto a submodule of 
 $(A_r^\FF)^{\unrhd \la}/(A_r^\FF)^{\rhd \la}$.   Therefore, for $b \in A_r^\FF$,  if $f^\la_\mft b = \sum_\mfs \beta_\mfs f^\la_\mfs$,  then
 $$
 x^* \c_\la d_\mft F_\mft  b = \sum_\mfs \beta_\mfs  x^* \c_\la d_\mfs F_\mfs  + y,
 $$
 where $y \in (A_r^\FF)^{\rhd \la}$.    Now for all $\mfs \in \Std_r(\la)$,   $F_\mfs \z k \la = F_\mfs$,  but
 $y \z k \la = 0$.    Thus multiplying by $\z k \la$ on the right gives
 $$
 x^* \c_\la d_\mft F_\mft  b = \sum_\mfs \beta_\mfs  x^* \c_\la d_\mfs F_\mfs,
 $$
 which shows that $\Span_\FF\{x^* \c_\la d_\mft F_\mft \suchthat \mft \in \Std_r(\la)\}$ is a right ideal and
 $f_\mft \mapsto x^* \c_\la d_\mft F_\mft $ determines an isomorphism of $\cell k \FF \la$ onto this right ideal.  Taking $x = 1$ yields statement (1),   and taking $x = d_\mfs$,  respectively $x = d_\mfs F_\mfs$, gives the isomorphisms in part  (2).  
 \end{proof}
 
 \begin{lem}\label{seminormal form} \mbox{}   Let $r \ge 0$, $\la, \mu \in \widehat A_r$,  $\mfs, \mft \in \Std_r(\la)$,  and
 $\mfu, \mfv \in \Std_r(\mu)$.
\begin{enumerate}[label=(\arabic{*}), font=\normalfont, align=left, leftmargin=*]
 \item  $f^\la_\mft F^\mu_{\mfu \mfv} = \delta_{\la, \mu} \langle f^\la_\mft,  f^\la_\mfu \rangle  f^\la_\mfv$. 
 \item  $F^\la_{\mfs \mft} F^\mu_{\mfu \mfv} = \delta_{\la, \mu} \langle f^\la_\mft,  f^\la_\mfu \rangle  F^\la_{\mfs \mfv}$. 
 \item  $\langle f^\la_\mft,  f^\la_\mfu \rangle \ne 0$ if and only if $\mft = \mfu$. 
 \item ${\langle f^\la_\mft,  f^\la_\mft \rangle}\inv  F^\la_{\mft \mft} = F_\mft$.  
 \item  The set of elements $E^\la_{\mfs \mft} = {\langle f^\la_\mfs,  f^\la_\mfs \rangle}\inv  F^\la_{\mfs \mft} $ for $\la \in \widehat A_r$ and $\mfs, \mft \in \Std_r(\la)$  is a complete family of matrix units with
 $$
 E^\la_{\mfs \mft} E^\mu_{\mfu \mfv} = \delta_{\la, \mu} \delta_{\mft, \mfu}  E^\la_{\mfs \mfv}, \quad \text{and}
 \quad E^\la_{\mfs \mft} E^\la_{\mft \mfs}  = E^\la_{\mfs \mfs}  =  F_\mfs.
 $$
 \end{enumerate}
 \end{lem}
 
 \begin{proof}  If $\la \ne \mu$, then  $f^\la_\mft F^\mu_{\mfu \mfv}  = f^\la_\mft  F_\mft  F_\mfu F^\mu_{\mfu \mfv}  = 0$, and similarly  $F^\la_{\mfs \mft} F^\mu_{\mfu \mfv} = 0$.    We have
 $f^\la_\mft F^\la_{\mfu \mfv} = \ \langle f^\la_\mft,  f^\la_\mfu \rangle  f^\la_\mfv$ by the definition of the bilinear form \eqref{defn of bilinear form}, since $F^\la_{\mfu \mfv}$ is a lift of 
 $\alpha_\la\inv( (f^\la_\mfu)^* \otimes f^\la_\mfv)$.   This proves part (1) and part (2) follows from
 \myref{Lemma}{lemma embedding of cell module in algebra}
 part (2).    If $\mft \ne \mfu$, then 
 $$
 \langle f^\la_\mft,  f^\la_\mfu \rangle = 
 \langle f^\la_\mft F_\mft,  f^\la_\mfu F_\mfu \rangle =  \langle f^\la_\mft F_\mft F_\mfu ,  f^\la_\mfu \rangle = 0.
 $$
 Suppose that  $\langle f^\la_\mft,  f^\la_\mft \rangle  = 0$ for some $\mft\in \Std_r(\la)$.  Then it follows from part (2) and the orthogonality of the elements $f^\la_\mfu$ which was just established, that
 $f^\la_\mft F^\mu_{\mfu \mfv} = 0$  for all $\mu$ and all $\mfu, \mfv \in \Std_r(\mu)$.   But by
 \myref{Corollary}{corollary seminormal basis},
 the identity element of $A_r^\FF$ is in the span of the set of  $F^\mu_{\mfu \mfv}$, so it follows that $f^\la_{\mft} = 0$, a contradiction.  This proves part (3).   By part (2),  $G_\mft = {\langle f^\la_\mft,  f^\la_\mft \rangle}\inv  F^\la_{\mft \mft}$ is an idempotent such that $G_\mft F_\mft = G_\mft$.  Since $F_\mft$ is a minimal idempotent, it follows that $G_\mft = F_\mft$, which proves part (4).   Part (5) follows from parts (2), (3),  and (4). 
 \end{proof}

\subsection{Restriction of the seminormal representations}
 For $0\leq s < r$ write
$A_r^\FF \cap  (A_s^\FF)'$ for the set of $x \in A_r^\FF$ that commute pointwise with $A_s^\FF$.
 
\begin{prop} \label{restriction seminormal}
 Let $1 \le s < r$.  Let $\nu \in \widehat A_r$  and $\mft \in \Std_r(\nu)$.  Write $\la = \mft(s)$,  $\mft_1 = \mft_{[0, s]} \in \Std_s(\la)$ , and $\mft_2 = \mft_{[s, r]} \in \Std_{s,r}(\nu \setminus \la)$.   
\begin{enumerate}[label=(\arabic{*}), font=\normalfont, align=left, leftmargin=*]
\item  Let $x \in A_s^\FF$  and suppose $f^\la_{\mft_1} x = \sum_{\mfs\in \Std_s(\la)} \alpha_\mfs f^\la_\mfs$.    We have that 
$$f^\nu_\mft x = \sum_{\mfs\in \Std_s(\la)} \alpha_\mfs  f^\nu_{\mfs \circ \mft_2}.$$ 
\item In particular,  for $\mu \in \widehat A_s$ and $\mfu, \mfv \in \Std_s(\mu)$,   
$$f_\mft^\nu  F^\mu_{\mfu \mfv} =  \delta_{\mu, \la} \delta_{\mft_1, \mfu}  \langle f^\la_{\mft_1},  f^\la_{\mft_1} \rangle 
f^\nu_{\mfv\circ \mft_2}.$$
\item  For $x \in A_r^\FF \cap (A_s^\FF)'$, 
$
f^\nu_\mft x = \sum_\mfs  r_\mfs f^\nu_{\mft_1 \circ s},
$
where the sum is over $\mfs \in \Std_{s,r}(\nu \setminus \la)$, and the coefficients depend only on $x$ and $\mft_2$, and are independent of $\mft_1$.    
\end{enumerate}
\end{prop}

\begin{proof}   We can embed the cell module $\cell r \FF \nu$ in $A_r^\FF$,  identifying $f^\nu_\mft$ with
$u^*_\mft F_\mft$, using
\myref{Lemma}{lemma embedding of cell module in algebra}.
  We can write $F_\mft = F_{\mft_1} F_{\mft_2}$,  where 
$F_{\mft_2} \in  A_r^\FF \cap (A_s^\FF)'$.    Thus
$$
f^\nu_\mft  x=  u^*_\mft F_\mft x = u^*_{\mft_2} F_{\mft_2}  u^*_{\mft_1} F_{\mft_1} x.
$$
Applying 
\myref{Lemma}{lemma embedding of cell module in algebra}
 part (1) to  $\cell s \FF \la$ we find that this equals
$$
u^*_{\mft_2} F_{\mft_2}  \sum_\mfs \alpha_\mfs  u^*_\mfs F_\mfs 
=
 \sum_\mfs \alpha_\mfs   u^*_{\mft_2} u^*_{\mfs}  F_{\mfs}  F_{  \mft_2} 
= \sum_\mfs \alpha_\mfs  u^*_{\mfs\circ \mft_2} F_{\mfs\circ \mft_2} 
 =   \sum_\mfs  \alpha_\mfs f^\nu_{\mfs \circ \mft_2}.
$$
This proves part (1) and part (2) is an immediate consequence. 
 We now consider part $(3)$. Given  
$x \in A_r^\FF \cap (A_s^\FF)'$,  we let 
$
f^\nu_\mft x = \sum_{\mfu\in \Std_r(\nu)}  r_\mfu f^\nu_\mfu$. 
Since $x$ commutes with $A^\FF_s$, 
$$
f^\nu_\mft x = f^\nu_\mft F_{\mft_1} x = f^\nu_\mft x  F_{\mft_1} = \sum_{\mfu}  r_\mfu f^\nu_\mfu  F_{\mft_1},
$$
which shows that $r_\mfu = 0$ unless $\mfu_{[0,s]}= \mft_1$.  Thus, we can rewrite this as
$f_\mft x = \sum_\mfs  r_\mfs  f_{\mft_1 \circ s},$ where now the sum is over $\mfs \in \Std_{s,r}(\nu \setminus \la)$.   It remains to show that the coefficients are independent of $\mft_1$.   If $\mfv \in \Std_s(\la)$, then
$$
\begin{aligned}
f^\nu_{\mfv \circ \mft_2} x &
=  f^\nu_{\mft_1\circ \mft_2} E^\la_{\mft_1 \mfv} x 
=  f^\nu_\mft E^\la_{\mft_1 \mfv} x =  f^\nu_\mft x E^\la_{\mft_1 \mfv}  
 =  \sum_\mfs r_\mfs  f^\nu_{\mft_1 \circ s}   E^\la_{\mft_1 \mfv}  
=   \sum_\mfs r_\mfs f^\nu_{\mfv\circ \mfs},
\end{aligned}
$$
where we have  applied  part (2).  
\end{proof}

\subsection{Dominance triangularity and  restriction rules for the Murphy basis}
For $0\leq s < r$, we now show that the Murphy   basis is compatible with restriction to the subalgebra 
$A _s \subseteq A_r$, and also to the subalgebra $A_r \cap A_s'$.  
 This is a first  step towards constructing   skew cell modules in the next section.  

\begin{lem}  Let $1 \le s < r$,  $\nu \in \widehat A_r$,  $\la  \in \widehat A_{s}$ and  $\mft  \in \Std_{s,r}(\nu \setminus \la )$.    Suppose that $y \in  \c_\la  A_{s} \cap  A_{s}^{\rhd \la }$.  Then there exist coefficients $r_\mathsf{z}  \in R$ such that 
$$
u^*_{\mft} y  \equiv \ \ \   
\sum_{\begin{subarray}c   \mathsf z \in \Std_r(\nu) \\   \mathsf z_{[s, r]} \rhd \mft \\  \mathsf z(s) \rhd \la   \end{subarray}}
  r_{\mathsf z}  u^*_{\mathsf z}
  \quad \mod A_r^{\rhd \nu}.
$$
\end{lem}

\begin{proof}   Recall that the element 
$\sum_{\la (s) \rhd \la }\z {s} {\la (s)}$ acts as the identity on the ideal $  A_{s}^{\rhd \la }$  and that  $y \in A_{s}^{\rhd \la }$.  
By assumption  $y \in \c_\la  A_{s}$ and $\stt(s)=\la$,  
we have    
  $$u^*_{\mft_{[s, j]}} y \in \c_{\mft(j)} A_j \subseteq A_j^{\unrhd \mft(j)}$$
for each $j$ with $s < j \le r$.  In fact,  if $y = \c_\la  x$,   then  $u^*_{\mft_{[s, j]}} y = u^*_{\mft_{[s, j]}} \c_\la x = \c_{\mft(j)} d_{\mft_{[s, j]}} x$, using \eqref{abstract branching compatibility for path}.
 It follows that $u^*_\mft y = u^*_\mft y  (\sum_{\la (j) \unrhd \mft(j)} \z j {\la (j)})$. 
   Finally, for 
$j < s$,  we can write $1 =  \sum_{\la (j) \in \widehat A_j}  \z j {\la (j)}$.   Arguing as in the proof of 
\myref{Theorem}{dominance triangularity},
 we have that 
$$
\begin{aligned}
u^*_\mft y   =  
\sum_{\begin{subarray}c  \mathsf z \in \Std_r \\ 
\mathsf z_{[s, r]} \rhd \mft\\ 
\mathsf z(s) \rhd \la   \end{subarray}}  
  u^*_\mft y F_{\mathsf z}   
 \equiv
\sum_{\begin{subarray}c  \mathsf z \in \Std_r (\nu)\\ 
\mathsf z_{[s, r]} \rhd \mft\\ 
\mathsf z(s) \rhd \la   \end{subarray}} 
u^*_\mft y F_{\mathsf z}   \quad \mod (A_r^\FF)^{\rhd \nu}.
\end{aligned}
$$
  Now, because the range of $F_{\mathsf z} $ on the cell module $\Delta_r^\FF(\nu)$ is 
$\FF\{ f_{\mathsf z}\} = \FF \{ u^*_{\mathsf z}  F_{\mathsf z}\}$, we have
$u^*_\mft y F_{\mathsf z}  =  \alpha_{\mathsf z}  u^*_{\mathsf z}  F_{\mathsf z} $ for some $  \alpha_{\mathsf z}  \in \FF$.  Thus
$$
u^*_\mft y  \equiv 
\sum_
{ \begin{subarray}c
  \mathsf z \in \Std_r(\nu)  \\
\mathsf z_{[s, r]} \rhd \mft \\
\mathsf z(s) \rhd \la  \end{subarray} }
 \alpha_{\mathsf z}  u^*_{\mathsf z} F_{\mathsf z}  \quad \mod (A_r^\FF)^{\rhd \nu}.
$$ 
By dominance triangularity 
(\myref{Theorem}{dominance triangularity})
we obtain 
$$
u^*_\mft y  = 
\sum 
 _{\begin{subarray}c
  \mathsf z \in \Std_r(\nu)  \\ 
\mathsf z_{[s, r]} \rhd \mft, \\
\mathsf z(s) \rhd \la   \end{subarray}} 
 r_{\mathsf z}  u^*_{\mathsf z}    
    +  y',
$$
with coefficients  $r_{\sf z} \in \FF$, and with $y' \in (A_r^\FF)^{\rhd \nu}$.    But since 
 $u^*_\mft y \in A_r^R$,  we have that $r_{\sf z} \in R$ and $y' \in A_r^{\rhd \nu}$. 
 \end{proof}
 
 The following is an immediate consequence of the lemma:
 
 \begin{prop}\label{jfgkdhkhjdfgdjhfgk} Let $1 \le s < r$,  $\nu \in \widehat A_r$,  $\la  \in \widehat A_{s}$ and  $\mft  \in \Std_{s,r}(\nu \setminus \la )$.   Let $x \in \c_\la  A_{s}$ and write
 $$
 x = \sum_{\mathsf s \in \Std_{s}(\la )}  \alpha_{\mathsf s}  u^*_{\mathsf s}  + y,
 $$
 with $y \in A_{s}^{\rhd \la }$.    Then there exist coefficients $r_{\mathsf z} \in R$, such that 
 $$
 u^*_\mft x \equiv \sum_{\mathsf s \in \Std_{s}(\la )}  \alpha_{\mathsf s}  u^*_\mft u^*_{\mathsf s}  + 
 \sum_{\begin{subarray}c
  \mathsf z \in \Std_r(\nu)  \\ 
\mathsf z_{[s, r]} \rhd \mft, \\
\mathsf z(s) \rhd \la   \end{subarray}}  
 r_{\mathsf z}  u^*_{\mathsf z}   \quad \mod A_r^{\rhd \nu}.
 $$
 \end{prop}

The following is a restriction rule for the Murphy type basis of a tower of diagram algebras.

\begin{cor}  \label{restriction rule for Murphy basis}
Let $r > 1$,  $\nu\in\widehat{A}_r$, and   $\mft \in \Std_r(\nu)$.
Let $1 \le s < r$ and write $\la  = \mft(s)$. 
Let $a \in A_{s}$ and suppose
 $$
 u^*_{\mft_{[0, s]}} a \equiv \sum_{\mfs \in \Std_{s}(\la )}  r_\sts u^*_\mfs  \mod A_{s}^{\rhd \la }.
 $$
 Then  there exist coefficients $r_{\mathsf z} \in R$ such that 
 $$
 u^*_\mft a \equiv   \sum_{\mfs \in \Std_{s}(\la )}  r_\sts  u^*_{\mft_{[s, r]}}u^*_\mfs   + 
 \sum _{\begin{subarray}c
  \mathsf z \in \Std_r(\nu)  \\ 
\mathsf z_{[s, r]} \rhd \mft, \\
\mathsf z(s) \rhd \la   \end{subarray}}   r_{\mathsf z}  u^*_{\mathsf z}  \quad \mod A_r^{\rhd \nu}.
 $$
 \end{cor}
 
 \begin{proof} Apply the previous Proposition with $x = u^*_{\mft_{[0, s]}} a$.  
 \end{proof}
 
 \begin{rem}
 \myref{Corollary}{restriction rule for Murphy basis}
 is an improvement of the restriction rule for ``path bases" obtained in 
 ~\cite[Proposition 2.18]{MR2774622}.
 \end{rem}
 
 We now consider restriction to $A_r \cap A_s' \subseteq A_r$.  
 
 \begin{prop} \label{Murphy restriction to commutant 1}
  Let $0 \le s < r$, $\nu \in \widehat A_r$, and $\stt \in \Std_r(\nu)$.  Write $\lambda = \stt(s)$.
 Let $x \in A_r \cap A_s'$.  Then there exist coefficients $r_\stw \in R$ such that
 $$
 u^*_\stt  x \equiv \sum_{\begin{subarray}c  \mathsf w \in \Std_r(\nu) \\ 
\mathsf w_{[0, s]} \unrhd  \stt_{[0, s]}  
  \end{subarray}} 
 r_\stw u^*_\stw  \quad \mod A_r^{\rhd \nu},
 $$
 \end{prop}
 
 \begin{proof}  Using 
 \myref{Theorem}{dominance triangularity},
write
 $
 m_\stt^\nu =  f_\stt^\nu  + \sum_{\stu \rhd \stt}  \alpha_\stu   f_\stu^\nu
 $, 
 with coefficients $\alpha_\stu \in \FF$, so that 
  $
 m_\stt^\nu  x=  f_\stt^\nu  x + \sum_{\stu \rhd \stt}  \alpha_\stu   f_\stu^\nu x
 $.  Now applying 
 \myref{Proposition}{restriction seminormal}
 part (3), we see that $m_\stt^\nu  x$ is an $\FF$--linear combination of seminormal basis elements $f_\stv^\nu$ with
 $\stv_{[0,s]} \unrhd \stt_{[0,s]}$.  Using
 \myref{Theorem}{dominance triangularity}
again,
  each such $f_\stv^\nu$ is an $\FF$--linear combination of Murphy basis elements $m_\stw^\nu$, with
  $\stw \unrhd \stv$, and thus $\stw_{[0,s]} \unrhd \stv_{[0,s]}  \unrhd \stt_{[0,s]}$.    Thus we get
  $$
 u^*_\stt  x = \sum_{\begin{subarray}c  \mathsf w \in \Std_r(\nu) \\ 
\mathsf w_{[0, s]} \unrhd  \stt_{[0, s]}  
  \end{subarray}} 
 r_\stw u^*_\stw   + y,
 $$
 with coefficients $r_\stw \in \FF$ and with $y \in A_r^{\rhd \nu} \otimes_R \FF$.    Since $u^*_\stt x \in A_r^R$,  the coefficients are necessarily in $R$ and $y \in  A_r^{\rhd \nu}$.  
 \end{proof}

    \section{Jucys--Murphy elements}  \label{section JM elements}
\def\JM{Jucys--Murphy }
We recall the definition and first properties of families of Jucys--Murphy elements for diagram algebras. The action of Jucys--Murphy elements on cell modules for diagram algebras was first  considered systematically in  Goodman--Graber \cite{MR2774622}, motivated by work of   Mathas ~\cite{MR2414949}.   
 In this section, we use 
 \myref{Theorem}{dominance triangularity}
 to  strengthen the  results of 
 ~\cite{MR2774622} by replacing the reverse lexicographic order on  tableaux with 
 the dominance order on  tableaux.  
 
 The treatment of \JM and the seminormal basis in \cite{MR2774622} proceeds as follows. From the definition of \JM  elements 
({\myref{Definitions}{defn additive JM} and \myrefnospace{}{defn mult JM})
one concludes that the \JM elements act triangularly with respect to reverse lexicographic order on the Murphy basis.  Therefore, \JM elements in the sense of ~\cite{MR2774622}  are also \JM elements in the sense of Mathas 
~\cite[Definition 2.4]{MR2414949}.  When the ground ring is a field  $\FF$  and Mathas' separation condition 
~\cite[Definition 2.8]{MR2414949} is satisfied -- which is true for standard examples of diagram algebras over the field of fractions of the generic ground ring --  then by ~\cite[Corollary 2.9]{MR2414949}, the algebras $A_r^\FF$ are split semisimple.  Moreover, 
following ~\cite[Section 3]{MR2414949} one can define a family
orthogonal idempotents  $F_\mft'$ labelled by $\mft \in \Std_r$, using interpolation formulas for the \JM elements.  By ~\cite[Proposition 3.11]{MR2774622},  the idempotents $F'_\mft$ coincide with the Gelfand--Zeitlin idempotents  $F_\mft$  for the tower $(A_r^\FF)_{r \ge 0}$.   Finally, one defines the seminormal bases as in 
\myref{Definition}{defn seminormal}.
 Since the $F_\mft$ are polynomials in the \JM elements,  it follows that transition matrix between the Murphy basis and the seminormal basis is unitriangular with respect to reverse lexicographic order. 

Here we reverse the logic.  We consider a tower of diagram algebras, i.e. a tower $(A_r)_{r \ge 0}$ satisfying \eqref{diagram 1}--\eqref{diagram compatibility}.    Then we already have the Gelfand--Zeitlin idempotents at our disposal and can define the seminormal bases as in 
\myref{Definition}{defn seminormal}.
 Moreover, the transition matrix between the Murphy basis and the seminormal basis is dominance unitriangular by 
 \myref{Theorem}{dominance triangularity}.
In case \JM elements in the sense of \myref{Definitions}{defn additive JM} and  \myrefnospace{}{defn mult JM} exist,  we will show that they act diagonally on the seminormal basis.  It then follows from  
\myref{Theorem}{dominance triangularity}
 that the \JM elements act triangularly on the Murphy basis with respect to dominance order.   Assuming Mathas' separation condition holds -- as it does in standard examples -- one can then take up the theory in
  ~\cite[Section 3]{MR2414949}, defining orthogonal idempotents by interpolation formulas using the \JM elements, and these must coincide with the Gelfand--Zeitlin idempotents by ~\cite[Proposition 3.11]{MR2774622}.  

For the rest of this section, assume $(A_r)_{r \ge 0}$  is a tower of diagram algebras over an integral domain $R$ with field of fractions $\FF$, 
satisfying \eqref{diagram 1}--\eqref{diagram compatibility}.

   \begin{defn}  \label{defn additive JM}
 We say that a family of  elements $\{L_r \mid  r \ge 1\}$,  is an {\sf additive family} of {\sf Jucys--Murphy  elements} if 
    the following conditions hold.
 \begin{enumerate}
  \item    For all $r \ge 1$, $L_r \in A_r$, $L_r$ is invariant under the involution of $A_r$,   and  
  $L_r$ commutes with $A_{r-1}$.  In particular,   $L_i L_j = L_j L_i$ for all $1\leq i\leq j \leq r$.
  \item For all $r \ge 1$ and   $\la \in\widehat{A}_r$,   there exists  $d(\la) \in R$  such that    $ L_1 +  \cdots + L_r$   acts as the scalar   $d(\la)$  on the cell module $\Delta^R_r(\la)$.
  For $\lambda \in \widehat{A}_{0}$, we let $d(\lambda)=0$.
    \end{enumerate}
    \end{defn}
  
\begin{eg}
 The  group algebras of symmetric groups, Temperley--Lieb, Brauer, walled Brauer,   and partition algebras
 all possess additive families of Jucys--Murphy elements.
 \end{eg}

   \begin{defn}   \label{defn mult JM}
    We say that a family of  elements $\{L_r \mid   r \ge 1\}$,  is a  {\sf multiplicative  family} of {\sf Jucys--Murphy  elements}
   if  the following conditions hold.
 \begin{enumerate}
  \item    For all $r \ge 1$,  
    $L_r$ is an invertible element of    $A_r$,
  $L_r$ is invariant under the involution $\ast$, 
   and  
  $L_r$ commutes with $A_{r-1}$.  In particular,   $L_i L_j = L_j L_i$ for all $1\leq i\leq j \leq r$.
  \item For all $r \ge 1$ and   $\la \in\widehat{A}_r$,   there exists  $d(\la) \in R$  such that    $ L_1   \cdots  L_r$   acts as the scalar   $d(\la)$  on the cell module $\Delta^R_r(\la)$.
  For $\lambda \in \widehat{A}_{0}$, we let $d(\lambda)=1$.
    \end{enumerate}
    \end{defn}
  
 \begin{eg}

  The  Hecke algebras of finite type $A$ and the Birman--Murakami--Wenzl algebra possess a multiplicative family of Jucys--Murphy elements.   
\end{eg}

  \begin{prop} \label{JM diagonal seminormal}  Assume that the tower $(A_r)_{r \ge 0}$ has additive or multiplicative \JM elements $L_i$.  Then the \JM elements act diagonally on the seminormal bases.  More precisely, there exist scalars $\kappa_{\mu \to \la} \in R$ associated to edges $\mu \to \la$ in $\widehat A$ such that for all $r$,  $\la \in \widehat A_r$, $\mft \in \Std_r(\la)$ and  $i \le r$, 
  $$
  f^\la_\mft L_i = \kappa_{\mft(i-1) \to \mft(i)}  f^\la_\mft.
  $$
  \end{prop} 
  
  \begin{proof}  We consider the case of additive \JM elements.  The proof for multiplicative \JM elements is nearly identical.   Let $\kappa_{\mu \to \la} = d(\la) - d(\mu)$, where $d(\cdot)$ is as in 
  \myref{Definition}{defn additive JM}.
  For $i \ge 1$ and $\mu \in \widehat A_i$, the sum 
  $L_1 + \cdots +L_i$ acts as the scalar $d(\mu)$ on $\cell i {} \mu$.  It follows from the restriction rule for the seminormal representations 
  (\myref{Proposition}{restriction seminormal})
that for all $r \ge i$, $\la \in \widehat A_r$ and $\mft \in \Std_r(\la)$,  $f^\la_\mft (L_1 + \cdots + L_i) = 
  d(\mft(i)) f^\la_\mft $.    Hence $f^\la_\mft L_i =  (d(\mft(i) - d(\mft(i-1))) f^\la_\mft$.
  \end{proof}

\begin{notation}  We will also write $\kappa_\mft(i) $ for $\kappa_{\mft(i-1) \to \mft(i)}$.  
\end{notation}

 \begin{thm} \label{JM dominance triangular}
   Assume that the tower $(A_r)_{r \ge 0}$ has additive or multiplicative \JM elements $L_i$.  Then the \JM elements act  triangularly with respect to dominance order on the Murphy basis of the cell modules, i.e.    for $r \ge 1$,  $\la \in \widehat A_r$ and $\mft \in \Std_r(\la)$,
 \begin{equation} \label{eqn JM dominance triangular}
 m^\la_\mft L_i = \kappa_\mft(i) m^\la_\mft  + \sum_{\mfs \rhd \mft} r_\mfs m^\la_\mfs,
 \end{equation}
 where the coefficients are in $R$.
 \end{thm}
 
 \begin{proof}  Using 
 \myref{Proposition}{JM diagonal seminormal}
 and 
 \myref{Theorem}{dominance triangularity},
 one obtains  
\eqref{eqn JM dominance triangular} with coefficients in $\FF$.    But since $L_i \in A_r^R$,  the coefficients are necessarily in $R$. 
 \end{proof}
 
  The dominance triangularity  results
 \myref{Theorems}{dominance triangularity}
  and 
  \myrefnospace{}{JM dominance triangular}
 apply to the following examples:  
 \begin{enumerate}[label=(\arabic{*}), ref=\arabic{*},leftmargin=0pt,itemindent=1.5em, series= JM examples]
 \item   The symmetric group algebras and the Hecke algebras of finite type $A$.  This example is discussed at length in  \myref{Section}{section: hecke} below.
 \end{enumerate}
 
 For the remaining examples,  dominance triangularity of the \JM elements, and of the transition matrix  from a path basis or a Murphy type basis to the seminormal basis, are new results.

 \begin{enumerate}[label=(\arabic{*}), ref=\arabic{*},leftmargin=0pt,itemindent=1.5em, resume= JM examples]
 \item The towers of  Tempreley--Lieb algebras,  partition algebras, Brauer algebras, and  Birman--Murakami--Wenzl algebras  \cite{EG:2012}.  
 \item The towers  of centralizer algebras on tensor space discussed in \myref{Example}{theoneexample}, \cite{BEG}.  
 \end{enumerate}

\section{Application to the Hecke algebras of the symmetric groups}  \label{section: hecke}

The arguments of \myref{Sections}{section seminormal} and \myrefnospace{}{section JM elements}, applied to the Hecke algebras of the symmetric groups,  result in some modest simplifications of the theory of these algebras, as presented, for example, in  ~\cite{MR1711316}. 

Let $S$ be an integral domain and $q\in S$ a unit.  The Hecke algebra $H_r(S; q)$ is the unital $S$--algebra with generators $T_1, \dots, T_{r-1}$  satisfying the braid relations and the quadratic relation
$(T_i - q) (T_i +1) = 0$.  For any $S$, the specialization  $H_r(S; 1)$ is isomorphic to $S \mathfrak S_r$.    
The generic ground ring for the Hecke algebras is the Laurent polynomial ring $R = \ZZ[\qbold, \qbold\inv]$, where $\qbold$ is an indeterminant.    Let $\FF$ denote $\mathbb Q(\qbold)$, the field  of fractions of $R$. 

We will write $H_r(\qbold)$ for $H_r(R; \qbold)$.   
The algebra $H_r(\qbold)$ has an $R$--basis $\{T_w \suchthat w \in \mathfrak S_r\}$, defined as follows:  if $w = s_{i_1} s_{i_2} \cdots s_{i_l}$ is a reduced expression for $w$ in the usual generators of $\mathfrak S_r$,  then $T_w = T_{i_1} T_{i_2} \cdots T_{i_l}$,  independent of the reduced expression.   
Define $T_w^* = T_{w\inv}$;  then $*$ is an algebra involution on $H_r(\qbold)$.

Let $\widehat {\mathfrak S}_r$ denote the set of Young diagrams or partitions of size $r$.   
 {\sf Dominance order} $\unrhd$ on $\widehat {\mathfrak S}_r$ is determined by $\lambda \unrhd \mu$ if    $\sum_{i = 1}^j  \lambda_i \ge \sum_{i = 1}^j  \mu_i $ for all $1\leq j\leq r$. 
{\sf Young's graph} or {\sf  lattice},  $\widehat{\mathfrak S}$, is the branching diagram with vertices 
$\widehat{\mathfrak S}_r$
on level $r$ and a directed edge $\lambda \to \mu$  if $\mu$ is obtained from $\lambda$ by adding one box.   We can identify {\sf standard tableaux} of shape $\lambda$ with directed paths on $\widehat{\mathfrak S}$ from $\varnothing$ to $\lambda$.  For $\lambda \in \widehat{\mathfrak S}_r$,  denote the set of standard tableaux of shape $\lambda$ by $\Std_r(\lambda)$.  

Let $\lambda$ be a Young diagram and let $\alpha = (i,j)$ be a box in $\lambda$.  The {\sf content} of $\alpha$ is $c(\alpha) = j-i$.   If $\stt \in \Std_r(\nu)$,  and $k \le r$, define $c_\stt(k) = c(\alpha)$,  where
$\alpha = \stt(k) \setminus \stt(k-1)$.   When $\stt$ is regarded as an array with the boxes of the Young diagram $\nu$ filled with the numbers from $1$ to $r$,  $c_\stt(k)$ is the content of the box containing the entry $k$.  

 For   $\lambda$ a partition of $ r$,  let $x_\lambda =  \sum_{w \in \mathfrak S_\lambda} T_w$.
 Let $\stt^\lambda$ be the row reading tableaux of shape $\lambda$.  For any $\lambda$--tableau $\stt$,  there is a unique $w(\stt) \in \mathfrak S_r$ with $\stt = \stt^\lambda w(\stt)$. 
   For $\lambda$ a partition of $r$ and $\sts,\stt\in\Std_r(\la)$, let 
   $$x^\lambda_{\sts \stt} =  (T_{w(\sts)})^*  x_\lambda T_{w(\stt)}.$$ 
   
   \begin{thm}[The Murphy basis, \cite{MR1327362}]  \label{thm Murphy basis of H n}  \mbox{}
 The set \ 
$
\mathcal X = \{  x_{\mfs  \mft}^\lambda  \suchthat  \lambda \in \widehat{\mathfrak S}_r \text{ and } \sts, \stt  \in \Std_r(\lambda) \}
$
is a cellular basis of $H_r(R; \qbold)$,   with respect to the involution $*$ and the partially ordered set
$(\widehat{\mathfrak S}_r, \unrhd)$.  
\end{thm}

We let $x^\lambda_\stt$ denote the basis element of the cell module $\cell r {} \lambda$ corresponding to $\stt \in \Std_r(\lambda)$, 
$x^\lambda_\stt =  x_\lambda T_{w(\stt)} + H_r^{\rhd \lambda}$.  
If $1 \le a \le i$, define
$$
T_{a, i} =  T_a T_{a+1} \cdots T_{i-1},
$$
and
$T_{i, a} = T_{a, i}^*  $.  
If $\lambda\vdash i-1$ and $\mu\vdash i$, with $\mu=\lambda\cup\lbrace (j,\mu_j)\rbrace$, let $a=\sum_{k=1}^{j}\mu_k$.  Define
 \begin{align} \label{branching factors for Hecke defn}
d_{\lambda\to\mu}=T_{a, i}\qquad\text{and}\qquad u_{\lambda\to\mu}=T_{i, a}\sum_{k=0}^{\lambda_j}T_{a,a-k}, 
\end{align}
 Given a  path
$\mft \in \Std_{r}(\nu )$, 
  $$
\varnothing = \stt(0) \to \stt(1) \to    \stt(2)\to   \dots \to  \stt(r-1)\to \stt(r) = \nu, 
 $$
define
 $$
d_\stt=  d_{\stt({r-1}) \to \stt( r)}  d_{ \stt( {r-2}) \to \stt( {r-1})} \cdots  d_{\stt(0) \to \stt( 1)}.
$$

\begin{prop}[\cite{EG:2012}]  \label{prop branching factors Hecke} \mbox{}
 \begin{enumerate}[label=(\arabic{*}), font=\normalfont, align=left, leftmargin=*]
 \item   For $\stt \in \Std_r(\nu)$, one has $d_\stt = T_{w(\stt)}$.
 \item  For $\lambda \to \mu$ in $\widehat{\mathfrak S}$, one has  $x_\mu d_{\lambda \to \mu} = u_{\lambda \to \mu}^*  x_\lambda$.
 \end{enumerate}
\end{prop}

Thus the Murphy cellular basis is recovered from the ordered products of branching factors:
$x_{\sts \stt}^\lambda =  (d_\sts)^*  x_\lambda d_\stt$.    Moreover, 
\myref{Remark}{remark factorization}   regarding factorization of representatives of the Murphy basis of cell modules applies to the Hecke algebra.   We want to stress that \myref{Proposition}{prop branching factors Hecke} is established by computation and does not rely on the connection, established in ~\cite{EG:2012}, between the branching factors $d_{\lambda \to \mu}$ and $u_{\lambda \to \mu}$ and cell filtrations of restricted and induced cell modules. 

Next, we recall the Jucys--Murphy elements for the Hecke algebras, see ~\cite{MR1711316}, Section 3.3 and Exercise 6, page 49.   The JM elements  in $H_n(\qbold)$ are defined by
$$
L_1 = 1  \text{ and } L_k =   \qbold^{1-k} T_{k-1} \cdots T_1 T_1 \cdots T_{k-1} \text{ for }  k > 1.  
$$

\begin{prop} \label{prop JM for Hecke} \mbox{}
 \begin{enumerate}[label=(\arabic{*}), font=\normalfont, align=left, leftmargin=*]
 \item  $L_1 L_2 \cdots L_r$ is in the center of $H_r(\qbold)$.
 \item  The elements $L_k$  are multiplicative JM elements in the sense of \myref{Definition}{defn mult JM}.
 \end{enumerate}
\end{prop}

See ~\cite[Section 3.3]{MR1711316}  for the proof. Again, the proof is computational and does not depend on deeper results on the Hecke algebras.   We will now list several properties of the Hecke algebras and afterwards sketch two logical routes through this material, both using the 
ideas of \myref{Sections}{section seminormal} and \myrefnospace{}{section JM elements}.

 \begin{enumerate}[label=(H\arabic{*}), ref=H\arabic{*},  series = HeckeAlgebras]
 \item \label{H ss} The Hecke algebras $H_r(\FF, \qbold)$ are split semisimple.
 \item \label{H br}   The branching diagram for the sequence $(H_r(\FF, \qbold))_{r \ge 0}$ of split semisimple algebras is Young's lattice.
 \item \label{H coherent}  The sequence of Hecke algebras over $R$ is restriction coherent, with $\cell {r-1} R {\lambda}$ appearing as a subquotient of $\Res^{H_r(\qbold)}_{H_{r-1}(\qbold)}(\cell r R \mu)$ if and only if $\lambda \subset \mu$. 
 \item \label{H coherent 2} More precisely,  statement (3) of \myref{Theorem}{theorem abstract Murphy basis}    holds.
 \end{enumerate}
 
 Since the sequence $(H_r(\FF, \qbold))_{r \ge 0}$  is a multiplicity free sequence of split semisimple algebras over $\FF$,  one can define Gelfand-Zeitlin idempotents as in \myref{Section}{GZ idempotents} and seminormal bases as in \myref{Definition}{defn seminormal} using the Gelfand-Zeitlin idempotents  and the Murphy basis.

\begin{enumerate}[label=(H\arabic{*}), ref=H\arabic{*},  resume = HeckeAlgebras]
\item \label{H sn}  The set $\{f^\la_\mft \suchthat \mft \in \Std_r(\la) \}$ is a basis of $\cell r \FF \la$ and the transition matrix between this basis and the Murphy basis is dominance unitriangular.
\item  \label{H JM action}   The JM elements act diagonally on the seminormal basis,   $f^\la_\stt L_k =  \qbold^{c_\stt(k)} f^\la_\stt $.
\item  \label{H JM action 2} The JM elements act triangularly on the Murphy basis,
$$
x^\la_\mft L_k = \qbold^{c_\stt(k)} x^\la_\mft  + \sum_{\mfs \rhd \stt} r_\mfs x^\la_\mfs.
$$
\end{enumerate}

We now sketch two logical paths through statements \eqref{H ss}--\eqref{H JM action 2}.  Of course, there are many logical arrangements of this material, and these are just two possibilities.

\medskip
\noindent{\em First path.}   One can first establish that the Hecke algebras $H_r(\FF, \qbold)$ are split semisimple.  One easy way to do this is as follows.   Let $\langle \cdot, \cdot \rangle$ denote the bilinear form on each cell module arising from the Murphy basis.   Let $\phi_\lambda$ be the determinant of the 
Gram matrix $[\langle m^\lambda_\sts,  m^\lambda_\stt \rangle ]_{\sts, \stt}$.  Then $\phi_\lambda \in R = \Z[\qbold, \qbold\inv]$.   Since the specialization of the Hecke algebra at $\qbold = 1$ is the symmetric group algebra over $\mathbb Q$, which is semisimple,  it follows that $\phi_\lambda(1) \ne 0$ and hence
$\phi_\lambda \ne 0$.   Now it follows from the general  theory of cellular algebras that $H_r(\FF, \qbold)$ is split semisimple (see ~\cite[Corollary 2.21]{MR1711316}).

The next step, which is more substantial, is to show that \eqref{H coherent} holds.  This is proved in 
~\cite{GKT-2014} or ~\cite{Mathas:2016}.  This implies that \eqref{H br} holds as well, using 
~\cite[Lemma 2.2]{MR2794027}, which is elementary. 

Taking into account \myref{Proposition}{prop branching factors Hecke}, we now have verified axioms
\eqref{diagram 1} through \eqref{diagram compatibility} for the Hecke algebras, as well as statements (1) and (2) of 
\myref{Theorem}{theorem abstract Murphy basis}.  A subtle point here is that we have not verified and do not need to verify at this point  that the $d$--branching coefficients actually arise from the cell filtrations of restricted cell modules, i.e. that statement (3) of \myref{Theorem}{theorem abstract Murphy basis} holds.

We are now entitled to plug into the arguments of \myref{Sections}{section seminormal} and \myrefnospace{}{section JM elements}, which give us the conclusions \eqref{H sn} through 
\eqref{H JM action 2}, but without revealing that the eigenvalues of the JM elements are $\qbold^{c_\stt(k)}$.    This additional information must come from an additional analysis, for example by extending the analysis of Okounkov and Vershik  ~\cite{MR1443185} to the Hecke algebras.   Moreover, we obtain 
\eqref{H coherent 2} from \myref{Corollary}{restriction rule for Murphy basis}.

\medskip
\noindent{\em Second path.}  The second path  follows  the analysis in 
~\cite[Chapter 3]{MR1711316} and its abstraction in ~\cite{MR2414949}.    Start with ~\cite[Theorem 3.32]{MR1711316}, which proves \eqref{H JM action 2}.    This implies that the JM elements $L_k$  are Jucys--Murphy elements in the sense of ~\cite{MR2414949} and moreover satisfy the ``separating condition" of 
~\cite[Section 3]{MR2414949} over $\FF$.   One can therefore define idempotents $F_\stt$ indexed by standard tableaux using interpolation formulas involving the JM elements as in ~\cite[Section 3.3]{MR1711316} or ~\cite[Section 3]{MR2414949}, and the seminormal basis of the cell modules by 
$f^\lambda_\stt = x^\lambda_\stt F_\stt$.  From the general theory in ~\cite[Section 3]{MR2414949} one obtains \eqref{H sn} and \eqref{H JM action}.   Moreover, the separating condition also implies \eqref{H ss}.  The restriction rule \eqref{H br}  results from actually computing the seminormal representations as in 
~\cite[Theorem 3.36]{MR1711316}.  We are now left with the task of verifying \eqref{H coherent} and 
\eqref{H coherent 2}.

Note that we have \eqref{diagram 1} --\eqref{diagram semisimplicity} as well as  \eqref{diagram compatibility} at our disposal as well as statements (1) and (2) of  \myref{Theorem}{theorem abstract Murphy basis}.  Moreover, by ~\cite[Proposition 3.11]{MR2774622}  the idempotents $F_\stt$ obtained from the JM elements have to coincide with the Gelfand-Zeitlin idempotents.   This is all we need to follow through the arguments of \myref{Section}{section seminormal}, and we end up with 
\myref{Corollary}{restriction rule for Murphy basis}, which implies \eqref{H coherent} and 
\eqref{H coherent 2}.

\begin{rmk}
The net result of this discussion is that if one uses the result from the literature on the restriction coherence of the tower of Hecke algebras, one can avoid some of the work involved with showing dominance triangularity of the JM elements;  and on the other hand, if one uses the results on dominance triangularity, one can avoid some of  the work involved in proving restriction coherence (which was only recently proven in  \cite{GKT-2014,Mathas:2016}). 
\end{rmk}

\section{Skew cell modules for diagram algebras}\label{sec:mainresult}
 In this section, we construct   skew cell modules for diagram algebras 
and provide integral bases of these modules indexed by skew tableaux. 
  We begin by constructing skew cell modules for $A_r \cap A_s'$  when $0 \le s \le r$, and afterwards introduce a final axiom which allow us to view these modules as $A_{r-s}$   modules.

Let $0 \le s < r$,  $\nu \in \widehat A_r$, and $\lambda \in \widehat A_s$.   Let $\stt^\lambda  \in \Std_s(\lambda)$  be maximal in $\Std_s(\lambda)$ with respect to the dominance order on paths.  Define
$$
\Delta(\nu; \rhd \lambda) = \Span_R\{ m^\nu_\stt \suchthat  \stt \in \Std_r(\nu) \text{ and } \stt(s) \rhd \lambda\},
$$  
and 
$$
\Delta(\nu; \stt^\lambda) = \Span_R\{ m^\nu_{\stt^\lambda \circ \stt }\suchthat  \stt \in \Std_{r-s}(\nu \setminus \lambda) \},
$$
both $R$-submodules 
 of $\cell r R \nu$.

  \begin{lem}  \label{skew lemma 1}
  $ \Delta(\nu; \rhd \lambda)$    and 
  $  \Delta(\nu; \stt^\lambda)+ \Delta(\nu; \rhd \lambda)$  are  $A_r \cap A_s'$  submodules of $\cell r R \nu$.  
 \end{lem}
 
 \begin{proof}  Follows from \myref{Proposition}{Murphy restriction to commutant 1}  and maximality of $\stt^\lambda$.  
 \end{proof}
 
Let $ \cell {} {} {\nu \setminus \lambda}  $ denote   the  $A_r \cap A_s'$  module
 $$
 \cell {} {} {\nu \setminus \lambda} =  ( \Delta(\nu; \stt^\lambda)+ \Delta(\nu; \rhd \lambda))/ \Delta(\nu; \rhd \lambda).
 $$

 \begin{lem} \label{skew lemma 2} \mbox{}
 \begin{enumerate}[label=(\arabic{*}), font=\normalfont, align=left, leftmargin=*]
 \item   $\cell {} {} {\nu \setminus \lambda}$ is, up to isomorphism, independent of the choice of the maximal element $\stt^\lambda$. 
 \item   $\Delta(\nu; \stt^\lambda) F_{\stt^\lambda} \subseteq \cell r \FF \nu$  is an $A_r \cap A_s'$ submodule, and 
  $\cell {} {} {\nu \setminus \lambda} \cong \Delta(\nu; \stt^\lambda) F_{\stt^\lambda} $.
  \item For any choice of $\sts \in \Std_s(\lambda)$,   
  $$ \cell {} {\FF} {\nu \setminus \lambda}  :=
  \cell {} {} {\nu \setminus \lambda} \otimes_R \FF  \cong
  \cell r \FF \nu F_\sts
  $$
  as $A_r^\FF \cap (A_s^\FF)'$--modules. 
 \end{enumerate}
 \end{lem}
 
 \begin{proof}  The first statement in part (2) follows from \myref{Lemma}{skew lemma 1}, because
  $ \Delta(\nu; \rhd \lambda) F_{\stt^\lambda}=0$.   Now we have an $A_r \cap A_s'$ module homomorphism $m \mapsto m F_{\stt^\lambda}$ from 
  $ \Delta(\nu; \stt^\lambda)+ \Delta(\nu; \rhd \lambda)$ to   $\Delta(\nu; \stt^\lambda) F_{\stt^\lambda} $
  with kernel $ \Delta(\nu; \rhd \lambda)$, which gives the isomorphism in part (2).    If $\stt_1$  and $\stt_2$  are two maximal elements in $\Std_s(\lambda)$,  then right multiplication by $E_{\stt_1 \stt_2}^\lambda  $  is an
  $A_r \cap A_s'$--module isomorphism from $\Delta(\nu; \stt_1) F_{\stt_1}$  to $\Delta(\nu; \stt_2) F_{\stt_2}$, which proves part (1).

 For part (3),   using
   \myref{Theorem}{dominance triangularity}    and  \myref{Proposition}{restriction seminormal} part (3),
      first note that given any
       $\sts,\bar\sts \in \Std_s(\lambda)$, we have that
  $\cell r \FF \nu {F_\sts } \cong \cell r \FF \nu {F_{\bar \sts}}$
as $R$--modules, with the isomorphism realized by right multiplication by $E_{\sts {\bar \sts}}^\lambda $.
This map preserves the   $A_r^\FF \cap (A_s^\FF)'$--module structure, since   $A_r^\FF \cap (A_s^\FF)'$ commutes with $E_{\sts {\bar \sts}}^\lambda$. 
In particular, we can set $\sts = \stt^\lambda$ without loss of generality.    Now, it follows from
\myref{Theorem}{dominance triangularity} that 
$$
\Delta^\FF(\nu; \rhd \lambda) := 
\Delta(\nu; \rhd \lambda)  \otimes_R \FF = \Span_\FF\{f_\stt^\nu \suchthat \stt \in \Std_r(\nu) \text{ and }
\stt(s) \rhd \lambda \}, 
$$
 and 
 $$ \Delta(\nu; \stt^\lambda) \otimes_R \FF + \Delta^\FF(\nu; \rhd \lambda)  =
\Span_\FF\{f_{\stt^\lambda \circ \stt}^\nu \suchthat \stt \in \Std_{r-s}(\nu \setminus \lambda)  \} + \Delta^\FF(\nu; \rhd \lambda).
$$
Arguing as for part (2), we have 
$$\cell {} {\FF} {\nu \setminus \lambda} \cong 
\Span_\FF\{f_{\stt^\lambda \circ \stt}^\nu \suchthat \stt \in \Std_{r-s}(\nu \setminus \lambda)  \} F_{\stt^\lambda}.
 $$
 Now one easily verifies  that $\Span_\FF\{f_{\stt^\lambda \circ \stt}^\nu \suchthat \stt \in \Std_{r-s}(\nu \setminus \lambda)  \} F_{\stt^\lambda} = \cell r \FF \nu F_{\stt^\lambda}$.  
 \end{proof}

  Now we introduce  one final axiom for towers of diagram algebras, which allows us to regard
  $\cell {} {} {\nu \setminus \lambda}$ as an $A_{r-s}$--module.
  
 \medskip
\begin{enumerate}[label=(D\arabic{*}), ref=D\arabic{*},  resume = DiagramAlgebras]
 \item  \label{subalgebra}  \label{diagram last}
  There is a an automorphism, $f_r$,   of order 2 of each $A_r $
   such that for each $s$ with $0 \le s  \le r$, 
   $f_r(A_{r-s})$ commutes with $A_{s} \subset A_r$.  
   \end{enumerate}
 \medskip
 
 Using this final axiom, we can define a homomorphism $\varphi_{r, s}: A_s \otimes_R A_{r-s} \to A_r$ by
 $$\varphi_{r, s}( a \otimes b) =  a f_r (f_{r-s}(b)),$$  where both $A_s$ and $A_{r-s}$ are regarded as subalgebras of $A_r$ via the usual embeddings.  We can restrict any $A_r$--module to  $A_s \otimes_R A_{r-s}$, by composing with $\varphi_{r, s}$.  

 \begin{rmk}
Condition \eqref{subalgebra} is satisfied by the group algebras of symmetric groups, the Hecke algebras,  and the Brauer, BMW, partition, and Temperley--Lieb algebras.  For each of these examples, 
 the involution $f_r$ is given by flipping a diagram or tangle through its vertical axis (whereas $\ast$ is given by flipping a diagram through its horizontal axis).    Also, in each of these genuine diagram or tangle algebras, one can define a tensor product operation, i.e. a homomorphism from $A_s \otimes_R A_{r-s}$ to 
$A_r$, which, on the level of diagrams or tangles, is just placing diagrams side by side.  Taking $f_r$ to be the flip through a vertical axis,  the homomorphism $\varphi_{r,s}$ defined above agrees with the homomorphism determined by placing diagrams side by side.   Note that $f_r\circ f_{r-s}$ is the shift operation on $A_{r-s}$ determined on the level of diagrams by adding $s$ vertical strands to the left of a diagram in $A_{r-s}$.
  \end{rmk}

Denote by  $\cell {r-s} {} {\nu \setminus \lambda}$ the $A_r \cap A_s'$ module   $\cell {}{}{\nu\setminus\lambda}$, when regarded as   an $A_{r-s}$--module by composing with the
homomorphism $f_r \circ f_{r-s} : A_{r-s} \to A_r \cap A_s'$.

 \begin{defn} \label{defn skew cell module}  
 Let $(A_r)_{r\geq 0}$ denote a tower of algebras satisfying conditions \eqref{diagram 1} -- 
 \eqref{subalgebra}.
Given $\la \in \widehat{A}_{s}$ and   $\nu \in \widehat{A}_{r}$,   the $A_{r-s}$--module 
$\Delta_{r-s}^R(\nu\setminus\lambda)$ described above is called the {\sf skew cell module} associated to $\lambda$ and $\nu$. 
 \end{defn}

 \newcommand{\coldomeq}{\unrhd_{\rm col}}
\newcommand{\coldom}{\rhd_{\rm col}}

 \begin{rmk}
 The definition of skew cell modules generalizes the definition of skew Specht modules for the tower of symmetric group algebras $(\ZZ \mathfrak S_r)_{r \ge 0}$.   The skew Specht module $S^{\nu\setminus \lambda}$ for a skew shape $  \nu \setminus\lambda$ is defined as the span of polytabloids of shape $ \nu \setminus\lambda$;  see  \cite[Section 4]{MR513828} for the definition of polytabloids, and \cite{MR528580} for the definition of skew Specht modules.   Let us regard $S^{\nu \setminus \lambda}$ as an $\mathfrak S_{\{s+1, \dots, r\}}$--module, where $|\lambda| = s$ and $|\nu| = r$.   Let $\stt_\lambda$ be the column reading standard tableau of shape $\lambda$, so $\stt_\lambda$ is the unique maximal tableau in column dominance order 
 $\coldomeq$ on standard tableaux of shape $\lambda$.  Consider the following subsets of  the Specht module $S^\nu$:
 $$
 S(\nu; \coldom \lambda) = \Span_\ZZ\{e_\stt \suchthat \stt \in \Std_r(\nu) \text{ and } \stt(s) \coldom \lambda\}
 $$
 and
 $$
 S(\nu; \stt_\lambda) = \Span_\ZZ\{e_{\stt_\lambda \circ \sts} \suchthat \sts \in \Std_{r-s}(\nu\setminus \lambda) \},
  $$
 where $e_\stt$ denotes a polytabloid.
 Using the Garnir relations \cite{MR513828}, one can verify that $ S(\nu; \coldom \lambda)  $  and 
 $S(\nu; \stt_\lambda)  +  S(\nu; \coldom \lambda)   $  
 are $\mathfrak S_{\{s+1, \dots, r\}}$ submodules of $S^\nu$, and the quotient of these modules is isomorphic to the skew cell module $S^{\nu \setminus \lambda}$.  A more general statement is proved in 
\cite[Theorem 3.1]{MR528580}.  Modulo the identification of cell modules of $\mathfrak S_r$ with classical Specht modules -- which involves both transpose of diagrams and twisting by the automorphism $s_i  \mapsto -s_i$, see \cite[Section 5]{MR1327362} -- this shows that the classical skew Specht modules agree with the skew cell modules defined here. 
 \end{rmk}

\begin{prop}\label{fdsakhjasldfhlhsdfsdfhl}
Given $\la \in \widehat{A}_{s}$, 
$\nu \in \widehat{A}_{r}$,   and an $A^\FF_{r-s}$--module $M$, 
 we have  
 \begin{align*}
 \Hom_{A^\FF_{ s}\otimes A^\FF_{r-s}}
(\Delta_{s}^\FF(\lambda) \otimes_\FF M, 
\Res^{A^\FF_{r}}_{A^\FF_{s}\otimes A^\FF_{r-s}}(\Delta_r^\FF(\nu ) )
 \cong 
\Hom_{A^\FF_{r-s}}
(M, \Delta_{r-s}^\FF(\nu\setminus\lambda)). 
  \end{align*}
\end{prop}

\begin{proof}  Identify $\Delta_{r-s}^\FF(\nu\setminus\lambda)$ with $\Delta^\FF_r(\nu) F_{\stt^\lambda}$, using  \myref{Lemma}{skew lemma 2} part (2).  
For $\varphi \in  \Hom_{A^\FF_{ s}\otimes A^\FF_{r-s}}
(\Delta_{s}^\FF(\lambda) \otimes_\FF M, 
\Res^{A^\FF_{r}}_{A^\FF_{s}\otimes A^\FF_{r-s}}(\Delta_r^\FF(\nu ) )$, define
$\overline \varphi$ by 
$$\overline \varphi(m) = \varphi(f^\lambda_{\stt^\lambda}\otimes m) = \varphi(f^\lambda_{\stt^\lambda}F_{\stt^\lambda}\otimes m)  = \varphi(f^\lambda_{\stt^\lambda}\otimes m) F_{\stt^\lambda}.
$$
It follows 
 that
 $\overline \varphi \in \Hom_{A^\FF_{r-s}}
(M, \Delta_{r-s}^\FF(\nu\setminus\lambda))$. 
  We have to check that  the map $\varphi \mapsto \overline \varphi$ is an isomorphism.  For injectivity, suppose $  
   \overline\varphi = 0
  $.  
   Then for all $m \in M$ and all
$x \in A_s^\FF$,   
$ \varphi(f^\lambda_{\stt^\lambda}x \otimes m) =  \overline \varphi(m) x = 0.$  
 Hence $ 
 \varphi = 0$  since    
 $f^\lambda_{\stt^\lambda} A_s^\FF =  \cell s \FF \lambda$.
 For surjectivity, 
let $\psi \in 
\Hom_{A^\FF_{r-s}}
(M, \Delta_{r}^\FF(\nu) F_{\stt^\lambda})$.  Define $\varphi$ by
$\varphi(f^\lambda_{\stt^\lambda} x \otimes m) = \psi(m) x = \psi(m) F_{\stt^\lambda}x $ for 
$x \in A_s^\FF$.    Then $\varphi$ is well-defined because $f^\lambda_{\stt^\lambda} x = 0 \Leftrightarrow F_{\stt^\lambda} x = 0$.    Now one can easily check that $\varphi$ is an $A_s^\FF \otimes A_{r-s}^\FF$--homomorphism and that $\overline \varphi = \psi$.  
    \end{proof}

Finally, let $(A_r)_{r\geq0}$  denote a  tower of algebras  satisfying \eqref{diagram 1}--\eqref{subalgebra}.  
 Let $\la \in \widehat{A}_s$, $\nu \in \widehat{A}_r$, $\mu \in \widehat{A}_{r-s}$ and define associated multiplicities 
\begin{align}\label{a defintion of sorts} 
A_{\lambda,\mu}^\nu = 
\dim_{\mathbb{\QQ}}
{\rm Hom}_{{A}^\QQ_{r-s}}
(\c_\mu {A^\QQ_{r-s}}  , \Delta^\QQ_{r-s}(\nu \setminus\lambda) )
\end{align}
\begin{align}\label{a defintion of sorts2} 
a_{\lambda,\mu}^\nu = \dim_{\mathbb{\QQ}}
{\rm Hom}_{{A}^\QQ_{r-s}}
(\Delta^\QQ_{r-s}(\mu) , \Delta^\QQ_{r-s}(\nu \setminus\lambda) ).
\end{align}
 We recall that $(\c_\mu  A _{r-s}  / (\c_\mu {A _{r-s}} \cap  A _{r-s} ^{\rhd \mu} )$ is isomorphic to the cell module $\Delta_{s}(\mu)$ and therefore $a_{\lambda,\mu}^\nu \leq A_{\lambda,\mu}^\nu$, by definition.

    \begin{rmk}
Consider the tower of the group algebras of symmetric groups $(\mathfrak{S}_r)_{r\geq0}$.  
 Recall that a {\sf partition}    $\mu $  is defined to be a finite weakly decreasing sequence   of non-negative integers. 
 We define the {\sf degree} of the partition $\la$ to be the sum,   $|\mu|$, over all non-zero terms in this sequence. 
 Recall  that $\widehat{\mathfrak{S} }_{r-s}$ is the set of partitions of degree $r-s$, see \cite{EG:2012}.  
 Given $\mu$ a partition of $r-s$,
  the module $\c_\mu  \mathfrak{S}^\QQ_{r-s} $ is isomorphic to the so-called Young permutation module, that is the module 
   obtained by induction from the subgroup $\mathfrak{S}_\mu$ 
 which stabilizes the set $\{1,\dots ,\mu_1  \}\times \{\mu_1+1,\dots ,\mu_1+\mu_2  \}\times \dots  \subseteq \{1,\dots,r-s\}$.  
Therefore, the coefficients defined in \eqref{a defintion of sorts} and \eqref{a defintion of sorts2} above are the skew--Kostka and Littlewood--Richardson coefficients, respectively \cite[Pages 311 and 338]{MR1676282}.      
    \end{rmk}
    
    \subsection{The stable Kronecker coefficients }
 Given
  $\lambda  =(\la_1,\la_2,\dots, \la_\ell)$ a partition and $n\in \mathbb{N}$ sufficiently large  we  set 
  $$\lambda_{[n]}=(n-|\lambda|, \lambda_1,\lambda_2, \ldots,\lambda_{\ell}).$$  
Given $n\in \mathbb{N}$, we recall that $\mathfrak{S}_n$ denotes the symmetric group on $n$ letters.   
The tower of  algebras $(\ZZ \mathfrak{S}_{n})_{n \geq 0}$  satisfy conditions 
\eqref{diagram 1}--\eqref{diagram last}.  
 Given $r  \in \tfrac{1}{2}\mathbb{N}$, we let $P_{2r}(n)$ denote the partition algebra  on $r$ strands with parameter $n \in \mathbb{N}$.  
The tower of  algebras $( P_{r}(n))_{r \geq 0}$  satisfy conditions 
\eqref{diagram 1}--\eqref{diagram last}.  
 The representation theories of the symmetric groups and partition algebras  are intimately related via a generalizations of   classical Schur--Weyl duality.  
 Through this duality, we obtain the following theorem.  

\begin{thm}\label{thm}
    Let 
$\lambda $ be a partition of degree $ s$, $\mu$ be  a partition of degree $r-s$, and 
$\nu$ be   a partition of degree less than or  equal to 
$ r $.    We have that
 \begin{align*}
 &\Hom_{\mathbb{Q}\Sym_n}
  (\Delta_{\mathbb{Q}\mathfrak{S}_n} (\lambda_{[n]})\otimes \Delta_{\mathbb{Q}\mathfrak{S}_n} (\mu_{[n]})
  ,\Delta_{\mathbb{Q}\mathfrak{S}_n} (\nu_{[n]}))   
 \\ \cong 
 &\Hom_{{P_{r-s}^{\mathbb{Q}}(n)}}
(\Delta_{P_{r-s}^{\mathbb{Q}}(n)}(\mu), \Delta_{P_{r-s}^{\mathbb{Q}}(n)}(\nu\setminus\lambda)).
  \end{align*}
  for $n$ sufficiently large ($n\geq 2r$ will suffice).  
 \end{thm}
\begin{proof}
This follows immediately from \myref{Proposition}{fdsakhjasldfhlhsdfsdfhl} and \cite[Corollary 3.4]{BDO15}.
\end{proof}

\noindent Given  $\lambda \vdash r-s$, $\mu \vdash s$ and $\nu \vdash r$  and $n\geq 2r$ we are interested in the multiplicities 
  $$
P_{\lambda,\mu}^\nu = 
\dim_{\mathbb{Q}}
{\rm Hom}_{{P}^{\mathbb{Q}}_{ r-s}(n)}
(\c_\mu {P^{\mathbb{Q}}_{ r-s}(n)}  , \Delta^{\mathbb{Q}}_{ r-s}(\nu \setminus\lambda) )
$$
$$
p_{\lambda,\mu}^\nu = \dim_{\mathbb{Q}}
{\rm Hom}_{{P}^{\mathbb{Q}}_{ r-s}(n)}
(\Delta^{\mathbb{Q}}_{ r-s}(\mu) , \Delta^{\mathbb{Q}}_{ r-s}(\nu \setminus\lambda) )$$
 for a cell  quasi-idempotent $\c_\mu \in {P^{\mathbb{Q}}_{ r-s} }(n)$.  
  By \myref{Theorem}{thm} and \cite{BDO15}, the coefficients $   {p}_{\lambda,\mu}^\nu$ are equal to the {\sf stable Kronecker coefficients}.
   These coefficients  have been described as   `perhaps the most challenging, deep and mysterious objects in algebraic combinatorics'  
   ~\cite{MR3589885}.
On the other hand, the coefficients $  P_{\lambda,\mu}^\nu $ do not seem to have been studied anywhere in the literature.
Motivated by the classical case, we ask the following questions:
\begin{enumerate}[label={$\bullet$},leftmargin=*,itemsep=-0.05em]
\item
 can one interpret the coefficients $ P_{\lambda,\mu}^\nu $ and  $p_{\lambda,\mu}^\nu $ in terms of the combinatorics of skew-tableaux for the partition algebra?   
\item  do there exist natural generalizations of   the {\sf semistandard} and {\sf lattice permutation} conditions in this setting? 
\item    do   the coefficients  $P_{\lambda,\mu}^\nu $ provide a first step towards understanding the stable Kronecker coefficients $p_{\lambda,\mu}^\nu $? 
\end{enumerate} 
The first two authors shall address these questions in an upcoming series of papers with Maud De Visscher. 
In particular, we use the above interpretation  to provide a positive combinatorial description of the stable Kronecker coefficients labelled 
 by an infinite family of  triples of partitions (including the Littlewood--Richardson coefficients, and Kronecker coefficients indexed by two two-row partitions as important examples).

    \begin{ack}
We would like to thank 
the Royal Commission for the Exhibition of 1851 
and 
EPSRC grant EP/L01078X/1 for   financial support. 
\end{ack}

\bibliographystyle{amsplain}

\bibliography{skew}

\end{document}